\documentclass[reqno,11pt]{amsart}
\usepackage[applemac]{inputenc}  

\usepackage{dsfont}
\usepackage{amssymb}
\usepackage{amsmath}
\usepackage{amsthm}
\usepackage{amsfonts}
\usepackage[dvips]{graphicx}
\usepackage{subfig}
\usepackage{geometry}
\usepackage{color}
\usepackage{hyperref}
\hypersetup{
colorlinks=true,
linkcolor=blue,
citecolor=blue
}

\numberwithin{equation}{section}

\newtheorem{prop}{Proposition}[section]

\newtheorem{lemma}{Lemma}[section]

\newtheorem{thm}{Theorem}[section]

\newcommand{\Es}{\ensuremath{\mathbb{E}}}
\newcommand{\MM}{\ensuremath{\mathcal{M}}}

\newcommand{\EN}{\ensuremath{\mathbb{N}}}

\newcommand{\ER}{\ensuremath{\mathbb{R}}}
\newcommand{\PP}{\ensuremath{\mathbb{P}}}
\newcommand{\z}{\ensuremath{\mathbf{z}}}
\newcommand{\zz}{\ensuremath{\boldsymbol{\zeta}}}

\newcommand{\be}{\begin{enumerate}} 
\newcommand{\ee}{\end{enumerate}}
\newcommand{\bi}{\begin{itemize}}
\newcommand{\ei}{\end{itemize}}
\newcommand{\dr}{\partial}
\newcommand{\dbn}[1]{\left\lVert #1 \right\rVert}

\newcommand{\vep}{\varepsilon}

\begin{document}

\title[From stochastic Zakharov to multiplicative NLS]{From stochastic Zakharov system to multiplicative stochastic nonlinear Schr\"odinger equation}

\author[Gr\'egoire Barru\'e]{Gr\'egoire Barru\'e$^{\scriptsize 1}$}
\author[Anne de Bouard]{Anne de Bouard$^{\scriptsize 2}$}
\author[Arnaud Debussche]{Arnaud Debussche$^{\scriptsize 3}$}



\maketitle

\begin{center} \small
$^1$ Univ Rennes, IRMAR,UMR 6625,  
F35000, France; \\
\email{gregoire.barrue@orange.com}
\end{center}

\begin{center} \small
$^2$ CMAP, CNRS, Ecole polytechnique, Institut Polytechnique de Paris\\
91128 Palaiseau, France; \\
\email{anne.debouard@polytechnique.edu}
\end{center}

\begin{center} \small
$^3$ Univ Rennes, IRMAR,UMR 6625,  
F35000, France; \\
\email{arnaud.debussche@ens-rennes.fr}
\end{center}

\vskip 0.1 in
\noindent
{\small 
{\bf Abstract}. 
We study the convergence of a Zakharov system driven by a time white noise, colored in space, to a multiplicative stochastic
nonlinear Schr\"odinger equation, as the ion-sound speed tends to infinity. In the absence of noise, the conservation of
energy gives bounds on the solutions, but this evolution becomes singular in the presence of the noise. To overcome this
difficulty, we show that the problem may be recasted in the diffusion-approximation framework, and make use of the perturbed
test-function method. We also obtain convergence in probability. The result is limited to dimension one, to avoid too much
technicalities. As a prerequisite, we prove the existence and uniqueness of regular solutions of the stochastic Zakharov system.

}
\vskip 0.1 in

\tableofcontents
\section{Introduction}
 
The Zakharov system was introduced in \cite{Zakharov1972CollapseOL} as a simplified model for Langmuir turbulence, a phenomenon resulting 
from the nonlinear coupling between Langmuir waves and perturbations of the ion density in a partially ionized plasma. The system couples a
Schr\"odinger equation for the slowly varying complex envelope $u$ of the electric field with a wave equation for the deviation $n$ of the ion density 
from its average. Related systems of equations have also been derived from the coupled Euler equation for the electrons and ions, and Maxwell
equations for the electric field (see e.g. \cite{berge1996perturbative, Sulem, texier2005wkb}).

After normalization of the quantities $u$ and $n$ with respect to physical constants (see e.g. \cite{masmoudi2005klein}), the system may be written as
\begin{equation}
\label{DZK}
\left\{\begin{array}{lll}
				i\dr_tu=-\Delta u+nu, \\[0.2cm]
				\varepsilon^2\dr_t^2n=\Delta(n+\vert u\vert^2) ,
\end{array}
\right.  
\end{equation}
where the remaining coefficient $\varepsilon$ is proportional to the inverse of the ion sound velocity.

A large number of mathematical studies have been devoted to the Cauchy problem for the system \eqref{DZK}, as well as its limit as
$\varepsilon$ goes to zero: note indeed that the system formally converges to the cubic nonlinear Schr\"odinger equation
\begin{equation}\label{NLS}
i\dr_tu +\Delta u+\vert u\vert^2u=0
\end{equation} 
in this limit.

The first mathematical results for the system \eqref{DZK} were obtained by Sulem and Sulem in \cite{SuSu79} where the authors proved global existence of weak solutions in the energy space in dimension $d=1,2,3$ using a Galerkin method, and global well-posedness in dimension one for 
more regular initial data. 
Added and Added in \cite{AdAd84}  improved the results of \cite{SuSu79} by showing global well-posedness in dimension 2 for small initial data,  using Brezis-Gallou\'et inequality.  A local well-posedness result was proved by Schochet and Weinstein in \cite{SchoWein86} with an existence time interval independent of $\varepsilon$, which allowed them to obtain the convergence as $\varepsilon$ goes to 0 to a solution of  \eqref{NLS}.  In \cite{AdAd88} Added and Added investigated the rate of convergence for small amplitude solutions, and highlighted the presence of initial layer effects
in the absence of a compatibility condition on the initial data. The optimal rates for this convergence were obtained in \cite{TsuOz91}. This limit was also studied for instance in \cite{MasNak08} where the authors were interested in the convergence of the Klein-Gordon-Zakharov system to the nonlinear Schr\"odinger equation.

Other local well-posedness results for regular data, as well as propagation of regularity may be found in \cite{TsuOz92}, 
and  local existence in the energy space was obtained in \cite{Collianderthese}. 
Bourgain and Colliander also used a low-high frequency decomposition method in \cite{CoBo96}, and the method was also used in \cite{GiTsuVe97} to get some refinements for local well-posedness in general space dimension.  More results about local and global well-posedness may be found for instance in \cite{Co98,CoHoTzi08,MasNak09,Pech01} for different dimensions or different nonlinearities.

Here, we are interested in a stochastic version of  the Zakharov system, in the presence of an additional damping term.  The system may be written as:
\begin{equation}\label{e0.1}
\left\{\begin{array}{lll}
				i\dr_tu=-\Delta u+nu ,\\[0.2cm]
				\varepsilon^2d(\dr_tn) +\alpha \varepsilon \dr_tn =\Delta(n+\vert u\vert^2)dt +\phi dW,
\end{array}
\right.
\end{equation}
where $\alpha>0$. This system is complemented with inital data, 
which belong to some Sobolev spaces that we will specify in the different sections of the paper.
The operator $\phi$ is a smoothing operator on $L^2(\ER^d)$, and the random process $W_t$ is a cylindrical Wiener process on $L^2(\ER^d)$,
so that $\phi W_t$ is a $\phi \phi^*$-Wiener process.  This stochastic perturbation corresponds to a spatially correlated white noise in time, 
and can be physically understood as random external fluctuations in the system under consideration, as for example thermal fluctuations.  
The damping is introduced to compensate the input of energy due to the noise. For a smaller damping, {\it i.e.} with $\varepsilon$ replaced by $\varepsilon^\gamma$, $\gamma>1$, we believe that the limit is ill defined. A larger damping could be treated with similar arguments as in our work. Note that numerical simulations indicate that indeed a smaller damping does not allow to take the limit as $\varepsilon$ tends to $0$ (see \cite{barrue2022approximation}).
The Zakharov system with damping terms and stochastic forcing is used for example in  \cite{GuFo06}, where the authors provide a numerical study of Langmuir turbulence on incoherent scatter spectra.

As in the deterministic case, we expect that the solution of \eqref{e0.1} converges in some sense to a stochastic Nonlinear Schr\"odinger equation,
as $\varepsilon$ tends to zero. The conservative version of this later equation with a multiplicative noise has been studied by Debussche and de Bouard in \cite{DbDB99} and \cite{DBDb03}, while Barbu,  R\"ockner and Zhang in \cite{BRZ14, zbMATH06568221, zbMATH06782935} used a different approach based on rescaling transformations to prove well-posedness results in both conservative and nonconservative cases. Brzezniak and Millet \cite{BrMi14} studied the stochastic Nonlinear Schr\"odinger equation  on a two-dimensional compact Riemannian manifold, with the use of stochastic Strichartz
estimates (see also \cite{zbMATH06968619, zbMATH07081468,zbMATH07259282}). The defocusing mass and energy critical cases
have been studied by Zhang \cite{zhang2023stochastic} and Fan, Xu and Zhao \cite{zbMATH07754960}.

As for the stochastic Zakharov system, Tsutsumi proved in \cite{Tsu2022} the global existence of $L^2$-solutions for a system of the form \eqref{e0.1},
but with an additional additive noise in the equation for $u$, in space dimension $d=1$. More recently, a more general system with a multiplicative noise
in the equation for $u$ was studied by Herr, R\"ockner, Spitz and Zhang, for $d=3$, in \cite{herr2023three}, using the rescaling approach.

The aim of the present paper is to prove the convergence in probability, as $\varepsilon$ goes to zero, of the solution of \eqref{e0.1}, in dimension
$d=1$, to the solution of the stochastic nonlinear Schr\"odinger equation:
$$idu_t=(-\dr_x^2u_t-\vert u_t\vert^2u_t-\frac{i}{2}u_tF)dt-u_t(\dr_x^2)^{-1}\phi dW_t,$$ 
where the last two terms in the equation correspond to the Stratonovitch noise \linebreak $u_t\circ(\dr_x^2)^{-1}\phi dW_t$.  

The main difficulty comes from the fact that the energy of System \eqref{e0.1}:
\begin{equation}
H(u,n)=\dbn{\dr_xu}_{L^2}^2 + \frac{1}{2}\left(\dbn{n}_{L^2}^2+\dbn{\varepsilon\dr_x^{-1}\dr_tn}_{L^2}^2\right)+\int_\ER n\vert u\vert^2dx
\end{equation}
is no longer preserved, and more importantly it has a singular evolution as $\varepsilon$ goes to 0.  Indeed, applying formally the It\^o formula to $H$, 
and using \eqref{e0.1}, terms  which are not controlled by the energy or terms of order $\varepsilon^{-2}$ appear. In order to pass to the limit 
in $\varepsilon$ despite
this singular evolution, we use a predictor-corrector method usually called the Perturbed Test Function method. The method allows us to pass to the limit in a martingale problem and to identify the limit equation, but we also use it to obtain bounds on a modified energy, allowing us to prove the tightness
of the approximating sequence.  This method was first introduced in a finite dimensional setting in \cite{PSV77}, and many examples of applications 
can be found in the book of Garnier, Fouque, Papanicolaou and Solna \cite{FGP07}.  It was generalized to the infinite dimensional case for instance in \cite{DBG12},\cite{DbV12} and used in  \cite{BDT2023}, \cite{DbDMV16}.

The perturbed test function method usually allows to prove convergence in law of solutions. However, due to the special form of the noise considered here, we are able to obtain convergence in probability. The application of the method requires to deal with rather regular solutions in space. Thus, 
despite the fact that the global existence of $L^2$-solution was proved in \cite{Tsu2022}, we need first to prove the global existence of regular
solutions of the \eqref{e0.1}, under spacial regularity assumptions on the noise.
Note that the well-posedness result that we prove is still true without the addition of the damping term,  namely for $\alpha=0$, but the convergence of the system to the stochastic Nonlinear Schr\"odinger equation requires $\alpha>0$. It is not clear at all whether the system \eqref{e0.1} has a limit 
as $\varepsilon$ goes to $0$ in the case $\alpha =0$.

The paper is organized as follows: in Section \ref{notations}, we introduce some notations, and state our main results.  In Section \ref{wellposed} we 
prove  the global existence of regular solutions, in the strong probabilistic sense. The method is an adaptation of Kato's method (see \cite{Ka87}), 
together with estimates
similar to those obtained in \cite{SuSu79}, and some details will be skipped. In Section \ref{driving}, we state our convergence problem in the diffusion approximation 
framework, and give results on the driving process. In Section \ref{sectionenergie}, we prove a modified energy estimate, by using the perturbed test function 
method. Section \ref{tightness} is devoted to the proof of the tightness of the $\vep$-dependent family of solutions. Finally, we prove the convergence to the stochastic NLS equation in  Section $6$, by first proving a weak convergence result, through the study of the martingale problem.
Since we are in a situation where the driving process is an Ornstein-Uhlenbeck process and all correctors can be computed explicitly,  we are able to have explicit martingales when studying 
the martingale problem. We take advantage of this to obtain the convergence in probability of Equation \eqref{e0.1} to Equation \eqref{Schrostoc}. Note that convergence in probability in a similar but simpler problem - the Ornstein-Uhlenbeck solves an equation of order one and 
the setting is in finite dimension - has been obtained in \cite{garnier22}.
Finally, technical results are gathered in the appendix.

\section{Notations and main  results}\label{notations}

Throughout the paper,  for $p\geqslant 1$, we denote by $L^p(\ER,\mathbb{C})$ the Lebesgue space of $p$ integrable $\mathbb{C}$-valued functions on $\ER$, endowed with its usual norm.  For $p=2$, $(\cdot,\cdot)$ is the inner product of $L^2(\ER,\mathbb{C})$ given by $$(f,g)=   Re  \int_\ER f(x)\bar{g}(x)dx,$$ where $Re$ denotes the real part of the integral.
Sometimes we will just write $L^p:=L^p(\ER)$ without any additional precision on the codomain.  In fact, the only $\mathbb{C}$-valued function in our problem is $u$, and all the others functions are $\ER$-valued.  For $s\in\ER$,  we use the Sobolev space $H^s:=H^s(\ER)$ of tempered distributions $f\in\mathcal{S}'(\ER)$ such that $(1+\xi^2)^\frac{s}{2}\hat{f}(\xi)\in L^2(\ER)$, $\hat{f}$ being the Fourier transform of $f$,  endowed with its usual norm.  Besides we will need the homogeneous Sobolev space $\dot{H}^s:=\dot{H}^s(\ER)$ of tempered distributions $f\in\mathcal{S}'(\ER)$ such that $\xi^s\hat{f}(\xi)\in L^2(\ER)$. 
 We also denote by $H^s_{loc}(\ER)=H_{loc}^s$ the space of distributions $f$ such that for every $\ell>0$, $f\in H^s([-\ell,\ell])$, the standard Sobolev space on the interval $[-\ell,\ell]$.  Finally,  if  $H,K$ are Hilbert spaces, $\mathcal{L}_2(H,K)$ is the space of Hilbert-Schmidt operators from $H$ to $K$. When $H=K$ we write $\mathcal{L}_2(H)$.

To introduce a mathematical description of the system \eqref{e0.1}, we consider a probability space $(\Omega,\mathcal{F},\PP)$, endowed with a filtration $(\mathcal{F}_t)_{t\geqslant 0}$. We also introduce a sequence $(\beta_k)_{k\in\EN}$ of real-valued, independent  Brownian motions associated to this filtration,
and a complete orthonormal system $(e_k)_{k\in\EN}$  of $L^2(\ER; \ER)$; then $W(t,x,\omega) = \sum_{k\in\EN} \beta_k(t,\omega) e_k(x)$ is
a cylindrical Wiener process on $L^2(\ER)$, and the series $\phi W(t,x) = \sum_{k\in\EN} \phi e_k(x) \beta_k(t)$ defines a $H$-valued Wiener process with covariance operator $\phi \phi^*$, as soon as $\phi\in\mathcal{L}_2(L^2(\ER),H)$.

The system \eqref{e0.1} is then rewritten as
\begin{equation}\label{equationdetaillee}
\left\{\begin{array}{lll}
				i\dr_tu=-\dr_x^2u+nu \\
				\dr_tn=\mu \\
				\varepsilon^2d\mu +\alpha\varepsilon \mu =\dr_x^2(n+\vert u\vert^2)dt +\phi dW(t).
\end{array}
\right.
\end{equation}

The next theorem gives results about existence and uniqueness of regular solutions for the system \eqref{equationdetaillee}, for fixed $\varepsilon >0$: we first state a local existence result 
for initial data in $H^2\times H^1$. This result is obtained by applying Kato's method and a contraction argument, using easy estimates
on the linear equations, that we list in section $2$. If the initial data and the noise are slightly more regular, then we also prove, thanks to an adequate
decomposition of the solution and using estimates similar to those obtained in \cite{SuSu79}, that the solution is more regular and
is globally defined. 

\begin{thm}\label{regular-solutions}
Let $\varepsilon >0$ be fixed. Let $\alpha \geqslant0$ and $\phi\in\mathcal{L}^2(L^2, L^2\cap\dot{H}^{-1})$.  Let $(u_0,n_0,n_1)\in H^2(\mathbb{R})\times H^1(\ER)\times L^2(\ER)\cap\dot{H}^{-1}(\ER)$. Then there exists a unique solution $(u,n, \mu)$ to \eqref{equationdetaillee}, with continuous paths with
values in $H^2(\mathbb{R})\times H^1(\ER)\times L^2(\ER)\cap\dot{H}^{-1}(\ER)$, such that $(u(0),n(0),\mu(0))$ $=(u_0,n_0,n_1)$. This solution is defined on a random interval $[0,\tau(\omega))$ where $\tau(\omega)>0$ is a stopping time such that 
$$\tau(\omega)=+\infty \qquad \text{ or }\qquad \lim_{t\to \tau(\omega)}\max\left(\dbn{u(t)}_{H^2},\dbn{n(t)}_{H^1}\right)=+\infty.$$
If moreover $\phi\in\mathcal{L}^2(L^2,H^1\cap\dot{H}^{-1})$, $u_0\in H^3(\ER)$, $n_0\in H^2(\ER)$ and $n_1\in H^1(\ER)\cap\dot{H}^{-1}(\ER)$, 
then $\tau(\omega)=+\infty$, a.s. and  $(u,n,\mu)$ has trajectories a.s. in $C(\ER^+;H^3(\mathbb{R})\times H^2(\ER)\times H^1(\ER))$.
\end{thm}

The result of the previous theorem may of course be generalized to random $\mathcal{F}_0$-measurable initial data, by conditioning on 
$\mathcal{F}_0$.

We now turn to the main result of the paper, which is the convergence of the solution as $\varepsilon$ tends to $0$, to the solution of the 
stochastic nonlinear Schr\"odinger equation. This result will require more regularity on the noise, and we assume from now on that
$\phi\in\mathcal{L}^2(L^2,H^1\cap\dot{H}^{-1})$. In order to state the result, we rewrite the system in a different form: consider an adapted
solution $Z^\varepsilon$ of the linear damped wave equation
\begin{equation}\label{eqZeps}
\varepsilon^2d(\dr_tZ^\varepsilon)+\alpha\varepsilon\dr_tZ^\varepsilon=\dr_x^2Z^\varepsilon dt+\phi dW_t,\; Z^\varepsilon(0)=\partial_tZ^\varepsilon(0)=0.
\end{equation}
Its  trajectories are a.s. time continuous with values in $H^2$ (see section $2$ for a more precise definition). 
Then, $(u,n,\mu)$ is a solution of \eqref{equationdetaillee} if and only
if $n=m+Z^\varepsilon$, with $(u,m)$ solution of 
\begin{equation} \label{systemetranslate}
\left\{\begin{array}{l}
i\dr_tu=-\dr_x^2u+(m+Z^\varepsilon)u, \\[0.2cm]
\varepsilon^2\dr_t^2m+\alpha\varepsilon \dr_tm=\dr_x^2(m + \vert u\vert^2).
\end{array}\right.
\end{equation}
Note that if $u_0 \in H^3(\ER)$, $m_0 \in H^2(\ER)$ and $m_1\in H^1(\ER)\cap\dot{H}^{-1}(\ER)$, then applying Theorem~\ref{regular-solutions},
with $\mathcal{F}_0$-measurable initial data given by $u_0$, $n_0=m_0$ and $n_1=m_1$, we deduce that there is
a unique global solution of the system \eqref{systemetranslate} with $u(0)=u_0$, $m(0)=m_0$ and $\partial_t m(0)=m_1$. The convergence result 
is then as follows:

\begin{thm}\label{thmfinal}
Let $\alpha>0$. For any $T>0$,  and for any $u_0\in H^3(\ER),m_0\in H^2(\ER)\cap\dot{H}^{-1}(\ER),m_1\in H^1(\ER)\cap\dot{H}^{-1}(\ER)$,  the process $(u^\varepsilon,m^\varepsilon)$ solution of the system \eqref{systemetranslate} with $u(0)=u_0$, $m(0)=m_0$, $\dr_tm(0)=m_1$ and $\phi\in\mathcal{L}_2(L^2(\ER),H^3(\ER)\cap \dot{H}^{-4}(\ER))$ satisfies: $(u^\vep)$ converges in probability in $C([0,T],H^s_{loc}(\ER))$
 for any $s<1$, to $u$ solution of 
\begin{equation}\label{Schrostoc}
idu=(-\dr_x^2u-\vert u\vert^2u-\frac{i}{2}uF)dt-u(\dr_x^2)^{-1}\phi dW_t,
\end{equation}
 where $F(x)=\sum_{k\in\EN}\big((\dr_x^2)^{-1}\phi e_k\big)^2(x)$.
\end{thm}

Note that, under the above assumptions on the initial data $u_0$, and on the operator $\phi$, the existence of a unique solution of \eqref{Schrostoc}
with a.s. $H^1$-valued continuous paths is a consequence of Theorems 4.1 and 4.6 in \cite{DBDb03}. Moreover, it is not difficult to prove that given
the assumption on the noise, the solution $u$ has actually $H^2$-valued continuous paths.

\section{Well-posedness for fixed $\vep$} \label{wellposed}

This section is devoted to the proof of Theorem \ref{regular-solutions}. We start with the proof of the local well-posedness result, and propagation
of regularity, then global existence for more regular initial data is proved in Subsection \ref{global}, adapting the arguments of \cite{SuSu79}.

\subsection{Local Well-Posedness} \label{local}
In this section,  we  prove local well-posedness of the problem \eqref{equationdetaillee}. Here $\varepsilon$ is fixed so there is no loss of generality 
in assuming that $\varepsilon=1$.  We start with a first result which is proved by a fixed point argument, adapting the method used by Kato 
in \cite{Ka87} for the Schr\"odinger equation. Then we state a second result which shows that we can recover more regularity on our solutions.  
This extra regularity will be needed in order to show afterwards that the problem \eqref{equationdetaillee} is actually globally well-posed.

	\begin{prop}\label{solutionH2}
Let $\alpha \geqslant0$ and $\phi\in\mathcal{L}^2(L^2, L^2\cap\dot{H}^{-1})$.  Let $(u_0,n_0,n_1)\in H^2(\mathbb{R})\times H^1(\ER)\times L^2(\ER)\cap\dot{H}^{-1}(\ER)$. Then there exists a unique solution $(u,n,\mu)$ to \eqref{equationdetaillee} with continuous valued paths in $H^2(\mathbb{R})\times H^1(\ER)\times L^2(\ER)\cap\dot{H}^{-1}(\ER)$ such that $(u(0),n(0),\mu(0))=(u_0,n_0,n_1)$. This solution is defined on a random interval $[0,\tau(\omega))$ where $\tau(\omega)>0$ is a stopping time such that $$\tau(\omega)=+\infty \qquad \text{ or }\qquad \lim_{t\to \tau(\omega)}\max\left(\dbn{u(t)}_{H^2},\dbn{n(t)}_{H^1}\right)=+\infty.$$
\end{prop}
\begin{proof}
Let us fix $(u_0,n_0,n_1)$ satisfying the assumptions of Proposition \ref{solutionH2} and rewrite the system \eqref{equationdetaillee} 
(with $\vep=1$) in the mild form:
\begin{equation}\label{integralform}
\left\{\begin{array}{ll}
u(t)=\Gamma_1(u,n) (t)\\[0.2cm]
(n(t),\mu(t))=\Gamma_2(u,n,\mu)(t),
\end{array} \right.
\end{equation}
with 
\begin{equation}
\label{Gamma1}
\Gamma_1(u,n)(t)= U(t)u_0-i\int_0^t U(t-s)n(s)u(s)ds
\end{equation}
and
\begin{eqnarray}
\label{Gamma2}
\Gamma_2(u,n,\mu)(t) & =& S_\alpha(t)(n_0,n_1) + \int_0^t S_\alpha(t-s) (0, \dr_x^2\vert u\vert^2(s))ds \\
		& & + \int_0^t S_\alpha(t-s)(0, \phi)dW(s).
\end{eqnarray}
Here, $U(t)=e^{it\dr_x^2}$ is the free Schr\"odinger group on $L^2(\ER)$, while $S_\alpha(t)$ is the semi-group associated with the linear damped
wave equation, that is $S_\alpha(t) (n_0,n_1)$ is the solution of 
$$
\left\{ \begin{array}{l} \dr_t n=\mu\\
\dr_t \mu +\alpha \mu =\dr_x^2 n
\end{array} \right.
$$
with $(n(0),\mu(0))=(n_0,n_1)$. Note that $S_\alpha$ has an explicit expression in Fourier variables, and is a contraction semi-group in 
$H^1(\ER)\times L^2(\ER) \cap \dot{H}^{-1}(\ER)$ (see Section  \ref{salpha} in the Appendix).
Now, let us define, for $T>0$, the spaces
$$
Z_T=\{ u\in L^{\infty}(0,T;H^2(\ER),\; \dr_t u \in L^{\infty}(0,T;L^2(\ER)) \}
$$
and
$$
Y_T= \{n\in L^\infty(0,T;H^1(\ER)), \; \dr_tn\in L^\infty(0,T;L^2(\ER))\}.
$$
Let $R_0>0$, and $T_0>0$ be fixed. Using the arguments of \cite{Ka87}, it is not difficult to see that for all $n\in B_{R_0}(Y_{T_0})$,
the application $u\mapsto \Gamma_1(u,n)$ is a contraction in $Z_T$, for $T>0$ sufficiently small, depending only on $\|u_0\|_{H^2}$,
and $R_0$. Denoting then by $\mathcal{T}(n)$ the solution of
\begin{equation} \label{Schrod}
\left\{\begin{array}{l}
i \dr_t u=- \dr_x^2 u+nu\\
u(0)=u_0,
\end{array}\right.
\end{equation}
that is obtained in this way on $[0,T_0]$, it is easy to see that, for $T$ sufficiently small, depending again only on $R_0$, there exists a constant 
$C>0$ such that if $u_1=\mathcal{T}(n_1)$, and $u_2=\mathcal{T}(n_2)$, with $n_1, n_2 \in B_{R_0}(Y)$, then
$$
\Vert \dr_t u_1-\dr_t u_2\Vert_{L^{\infty}(0,T;L^2)} \leqslant C T (\|n_1-n_2\|_{Y_T} +\|u_1-u_2\|_{L^{\infty}(0,T;H^1)}).
$$
Moreover, using again equation \eqref{Schrod}, 
$$
\dbn{u_1-u_2}_{L^\infty(0,T;H^2)}  \leqslant C\dbn{n_1-n_2}_{Y_T},$$
where $C$, and $T$ small enough, depend only on $R_0$.

Now, setting $\Gamma_3(n,\mu)=\Gamma_2(\mathcal{T}(n),n,\mu)$ (where we recall that $\Gamma_2$ is defined in \eqref{Gamma2})
it is easily seen that $\Gamma_3$ is contracting in $B_R(Y_T\times H_T)$, where $H_T=L^{\infty}(0,T;L^2\cap \dot{H}^{-1})$, provided
$T$ is small enough, depending only on $R$. On the other hand, by \cite[Theorem 6.10]{DPZ14}, the stochastic convolution
$\int_0^{\cdot} S_\alpha (\cdot -s)(0,\phi)dW(s)$ is a.s. in $Y_{T_0}\times H_{T_0}$, and it follows that the ball $B_R(Y_T\times H_T)$ is preserved 
by $\Gamma_3$, provided the random constant $R$ is chosen sufficiently large (depending on the $Y_{T_0}\times H_{T_0}$-norm
of the stochastic convolution, $\|n_0\|_{H^1}$ and $\|n_1\|_{L^2}$) and $T$ is chosen small enough, depending on $R$.
The conclusion of Proposition \ref{solutionH2} follows by classical arguments.
\end{proof}

The next proposition  shows that if the initial data is more regular, then the solution of the system is also regular, as long as it exists in the sense of 
Proposition  \ref{solutionH2}.

\begin{prop}\label{solutionH3}
Let $\phi\in\mathcal{L}^2(L^2,H^1\cap\dot{H}^{-1})$ and let $(u,n,\mu)$ be the solution of \eqref{equationdetaillee} given by Proposition \ref{solutionH2},  and $\tau$ the corresponding stopping time.  If $u_0\in H^3(\ER)$, $n_0\in H^2(\ER)$, and $n_1\in H^1(\ER)\cap\dot{H}^{-1}(\ER)$, then  $(u,n,\mu)$ has a.s. continuous trajectories on $[0,\tau)$ with values in $H^3(\mathbb{R})\times H^2(\ER)\times H^1(\ER)\cap \dot{H}^{-1}(\ER)$.
\end{prop}

\begin{proof}
The proof is obtained by classical arguments: differentiating the system \eqref{equationdetaillee} with respect to $x$ gives a linear (non homogeneous)
equation for $(\dr_x u,\dr_x n,\dr_x \mu)$. Using then the same kind of fixed point arguments as in the proof of Proposition \ref{solutionH2}, it is easy
to prove that this later system has a unique solution in $Z_T\times Y_T \times H_T$, for $T$ sufficiently small, depending only on 
$\max(\|u\|_{Z_{T_0}}, \|n\|_{Y_{T_0}})$, where $T_0<\tau(\omega)$ is fixed, so that the argument may be iterated up to any time less 
than $\tau(\omega)$.
\end{proof}

\subsection{Global Well-Posedness} \label{global}

In this section,  we follow the arguments in \cite{SuSu79} to get a priori estimates showing that the solution is global.  We are still considering 
the system \eqref{equationdetaillee},
but now we mention the dependence  on $\varepsilon$ in the estimates. We will see in particular that those are not uniform in $\varepsilon$.
We prove below the following proposition.

\begin{prop}\label{GWP}
Let $\alpha\geqslant 0$, $T>0$ and let $\phi\in\mathcal{L}^2(L^2,H^1\cap\dot{H}^{-1})$.
For initial data $u_0\in H^3$, $n_0\in H^2$ and $n_1\in H^1\cap\dot{H}^{-1}$,  there exists almost surely a unique solution $(u,n,\mu)\in L^\infty([0,T], H^3\times H^2\times H^1\cap \dot{H}^{-1})$ for the system \eqref{equationdetaillee}.
\end{prop}

\begin{proof}
In order to get our estimates,  we decompose $n=m+ Z^\varepsilon$ where $Z^\varepsilon$  satisfies 
\begin{equation}\label{Zepsilon}
\begin{split}
\varepsilon^2d(\dr_tZ^\varepsilon)&=(-\alpha\varepsilon\dr_tZ^\varepsilon+\dr_x^2Z^\varepsilon)dt +\phi dW_t,\\
Z^\varepsilon(0)&= \dr_tZ^\varepsilon(0)=0.
\end{split}
\end{equation} 
Note that $(Z^\varepsilon, \dr_t Z^\vep)$ has a.s. continuous trajectories with values in $H^2\times H^1\cap\dot{H}^{-1}$ (again by 
\cite[Theorem 6.10]{DPZ14}). Also, $(u,n,\mu)$ is solution of the system \eqref{equationdetaillee} if and only if $(u,m)$ is solution of the system
\eqref{systemetranslate}.
Thus, under the assumptions of Proposition \ref{GWP}, the solution $(u,m)$ given by Proposition \ref{solutionH3}  satisfies a.s. 
$\dr_t m \in C([0,\tau);H^1\cap \dot{H}^{-1})$, and we may define 
\begin{equation}\label{definitionV}
V=-\dr_x^{-1}\dr_tm \in C([0,\tau);H^2), \; \text{a.s.}
\end{equation}
for which
\begin{equation}\label{eqV}
\varepsilon^2\dr_tV+\alpha\varepsilon V+\dr_x(m+\vert u\vert^2)=0.
\end{equation} 

Now, let us define, with the above notations, the mass and energy
	\begin{equation}\label{defNH}
	N(u)=\dbn{u}_{L^2}^2 \text{ and }\quad H(u,m)=\dbn{\dr_xu}_{L^2}^2 + \frac{1}{2}(\dbn{m}_{L^2}^2 + \dbn{\varepsilon V}_{L^2}^2) + \int_\ER m\vert u\vert^2dx.
	\end{equation}
The evolution of those quantities allows us to get first the following estimate.

\begin{lemma}\label{GronwallH}
Under the above assumptions, the local solution $(u,m)$ of the system \eqref{systemetranslate} satisfies: for any $T>0$, there exists a random 
constant $C_\varepsilon(\omega)>0$ depending only on $T$, $\|u_0\|_{H^1}$, $\|n_0\|_{L^2}$ and $\|Z^\vep\|_{L^\infty(0,T;H^2)}$, such that
\begin{equation}\label{estimationenergie}
\frac{1}{2}\dbn{\dr_xu}_{L^2}^2 + \frac{1}{4}\dbn{m}_{L^2}^2 + \frac{1}{2}\dbn{\varepsilon V}_{L^2}^2 \leqslant C_\varepsilon, \; \text{ a.s.}
\end{equation}
\end{lemma}

\begin{proof}
Using the fact that $m$ is real valued, it is easy to see (taking the inner product in $L^2(\ER;\mathbb{C})$ of the first equation of 
\eqref{systemetranslate} with $iu$)
that the mass $N(u)$ is conserved. As for the energy, under the above regularity assumptions, we may take the inner product of the first
equation of \eqref{systemetranslate} with $\dr_t u$, and integrate by parts to get
\begin{equation}\label{en1}
\dr_t\dbn{\dr_xu}_{L^2}^2 +\int_\ER (m\dr_t\vert u\vert^2 +  Z^\varepsilon\dr_t\vert u\vert^2 )\, dx=0.
\end{equation}
On the other hand, taking the inner product of equation \eqref{eqV} with $V$, integrating by parts the term $(\dr_x (m+|u|^2),V)$ and adding the
result to equation \eqref{en1} gives
\begin{equation}\label{H}
\dr_tH(u,m) + \int_\ER Z^\varepsilon\dr_t\vert u\vert^2dx =-\alpha\varepsilon\dbn{V}_{L^2}^2.
\end{equation}
Note that Young and Galiardo-Nirenberg inequalities, together with the conservation of mass yield
$$
\Big\vert \int_\ER m\vert u\vert^2dx\Big\vert \leqslant \frac{1}{4}\dbn{m}_{L^2}^2+ 4N(u_0)^\frac{3}{2}\dbn{\dr_xu}_{L^2}
$$
(see \cite[Lemma 2]{SuSu79}), while $\dr_t|u|^2=-2 Im (\bar u \dr_x^2 u)$ by the first equation of \eqref{systemetranslate}, so that
integrating by parts,
$$
\int_\ER Z^\vep \dr_t |u|^2 \, dx = \int_\ER \dr_xZ^\varepsilon Im (\bar{u}\dr_xu) \, dx 
$$ 
and the later term is bounded by $\|\dr_x Z^\vep \|_{L^\infty} N(u_0)+ \|\dr_x u\|_{L^2}^2$. The conclusion of Lemma \ref{GronwallH} is then obtained by
integrating \eqref{H} between $0$ and $t$, using the above bounds and Gr\"onwall Lemma.
\end{proof}

The next equalities are very similar to those stated in \cite{SuSu79}, and we only sketch their proofs.  

\begin{lemma} \label{L.5.1.4}
Under the above assumptions, the local solution $(u,m)$ of the system \eqref{systemetranslate} satisfies
\begin{align}
\label{eq219}
& \dr_t\big(\dbn{\varepsilon \dr_t V}_{L^2}^2 + \dbn{\dr_tm}_{L^2}^2\big) + 2 \int_\ER\dr_t^2m \,\dr_t\vert u\vert^2dx+2\alpha\varepsilon\dbn{\dr_tV}_{L^2}^2=0, \\
\label{eq220}
& \dr_t\dbn{\dr_tu}_{L^2}^2= 2  Im \int_\ER u\dr_t\bar{u}\,(\dr_tm + \dr_tZ^\varepsilon) \,dx, \\
\label{eq221}
& -\dr_t\dbn{\dr_t \dr_x u}_{L^2}^2 = \int_\ER (m+ Z^\varepsilon)\dr_t\vert \dr_tu\vert^2dx + \int_\ER\dr_t(m+ Z^\varepsilon)\big( \dr_t^2\vert u\vert^2-2\vert\dr_tu\vert^2\big)dx.
\end{align} 
\end{lemma}
\begin{proof}
Equality \eqref{eq219} is obtained by taking the time derivative of \eqref{eqV}, using integration by parts and the fact that $\dr_x\dr_tV=-\dr_t^2m$.
As for \eqref{eq220}, it suffices to differentiate the first equation of \eqref{systemetranslate} in time, take the scalar product of the resulting equation
with $\dr_t u$, and integrate by parts.
Finally, \eqref{eq221} is obtained by differentiating in time the equation $\dr_x^2 u =(m+Z^\vep)u -i\dr_t u$, taking the inner product with $\dr_t^2 u$,
integrating by parts, and using the fact that $Re (u\dr_t^2 \bar u)=\frac12(\dr_t^2|u|^2-2|\dr_tu|^2)$.
\end{proof}

Now, by gathering \eqref{eq219}-\eqref{eq221} and integrating in time we obtain: 
\begin{equation} \label{grosseestimation}
\begin{split}
  2\dbn{\dr_tu}_{H^1}^2  &+ \dbn{\varepsilon\dr_tV}_{L^2}^2 + \dbn{\dr_tm}_{L^2}^2 + 2\int_\ER (\dr_tm \, \dr_t\vert u\vert^2 +m\vert\dr_tu\vert^2) \, dx
 +2\alpha\varepsilon\dbn{V}_{L^2([0,t],L^2)}^2 \\ 
= & \; C(u_0,n_0,n_1, Z^\vep(0)) +\int_0^t\!\!\int_\ER \big(6\dr_tm\,\vert\dr_tu\vert^2 + 4 Im (u\, \dr_t\bar{u}\,\dr_tm)\big) dxds \\
&+4 Im \int_0^t  \!\!\int_\ER u \, \dr_t\bar{u} \,\dr_tZ^\varepsilon dxds - 2\int_\ER Z^\varepsilon\vert\dr_tu\vert^2 dx\\
&-2\int_0^t\!\!\int_\ER \dr_tZ^\varepsilon\big(\dr_t^2\vert u\vert^2-3\vert\dr_tu\vert^2\big)dxds. \\
\end{split}
\end{equation}
In what follows, $
C$ denotes various constants that may depend on $\omega, T, u_0, n_0, n_1$ and $\vep$.
Note that the first term on the second line of \eqref{grosseestimation} is finite, thanks to the assumptions on the initial data. As for the second term, 
using the fact that $\dr_tm=-\dr_xV$, integrating by parts and using Lemma \ref{GronwallH}, it is bounded by $C\int_0^t \|\dr_t u\|_{H^1}^2 ds$.
The last term of the second line is in turn bounded by $C\int_0^t (\|\dr_t u\|_{H^1}^2 +\|\dr_tm\|_{L^2}^2) \,ds$, by conservation of mass.

We now estimate the terms in equation \eqref{grosseestimation} involving the stochastic convolution $Z^\varepsilon$.
Using the fact that $(Z^\varepsilon,\dr_tZ^\varepsilon)\in L^\infty([0,T],H^2) \times L^\infty([0,T],H^1)$ a.s.,
the first term involving $Z^\vep$ is bounded by $C(1+\int_0^t \|\dr_t u\|_{L^2}^2 ds)$.
For the third term involving $Z^\vep$, we use the fact that  $\dr_t |u|^2=-2 Im (\bar u \,\dr_x^2 u)$, differentiate in time, and integrate the term
by parts to get
$$
\int_0^t \int_\ER \dr_t Z^\vep \dr_t^2|u|^2ds \, ds \leqslant 2\int_0^t \|u\|_{H^1} \|\dr_t Z^\vep\|_{H^1} \|\dr_t u\|_{H^1} ds.
$$
Lemma \ref{GronwallH} allows then to bound also the third term involving $Z^\vep$ in \eqref{grosseestimation} by 
$C_\varepsilon(1+\int_0^t \|\dr_t u\|_{L^2}^2 ds)$.

There remains to bound the terms in  \eqref{grosseestimation} that are not integrated in time. Note that, again by  the relation $\dr_tm=-\dr_x V$,
and integrating by parts,
$$
\int_\ER \dr_t m \,\dr_t|u|^2 dx \leqslant \|V\|_{L^2} \|u\|_{H^1} \|\dr_t u\|_{H^1} 
\leqslant \frac12 \|\dr_t u\|_{H^1}^2 +C_\varepsilon.
$$
On the other hand, 
\eqref{eq220} and Lemma \ref{GronwallH} easily imply
$$
\|\dr_t u\|_{L^2}^2 \leqslant C_\varepsilon (1+\int_0^t \|\dr_t u\|_{L^2}^2 +\|\dr_t m\|_{L^2}^2 ds),
$$
from which we deduce 
$$
\int_\ER m|\dr_t u|^2 dx \leqslant \|m\|_{L^2}\|\dr_t u\|_{L^2} \|\dr_t u\|_{H^1} 
\leqslant \frac12 \|\dr_t u\|_{H^1}  + C( 1+\int_0^t \|\dr_t u\|_{L^2}^2 +\|\dr_t m\|_{L^2}^2 ds).
$$

Gathering the above estimates allows to conclude that 
\begin{equation}\label{boundglobalwp}
\dbn{\dr_tu}_{H^1}^2 + \dbn{\varepsilon\dr_tV}_{L^2}^2 + \dbn{\dr_t m}_{L^2}^2 \leqslant C_\varepsilon(1+ \int_0^t\dbn{\dr_tu(s)}_{H^1}^2 + \dbn{\dr_tm(s)}_{L^2}^2ds),
\end{equation}
so that by Gr\"onwall Lemma, the quantities $\dbn{\dr_tu}_{H^1}, \dbn{\dr_tm}_{L^2}$ and $\dbn{\varepsilon\dr_t V}_{L^2}$ are a.s. bounded 
on $[0,T]$.

Next, equation \eqref{eqV} and Lemma \ref{GronwallH} imply $\|m\|_{H^1}\leqslant C_\varepsilon$, while the first equation of \eqref{systemetranslate}
and Lemma  \ref{GronwallH} give a bound on $\|\dr_x^2 u\|_{H^1}$, showing that $u \in L^\infty(0,T;H^3)$, a.s.
The conclusion of the proof of Proposition \ref{GWP} follows then from Proposition \ref{solutionH3}, and the blow up criterion in Proposition 
\ref{solutionH2}.
\end{proof}

\section{The driving process and its generator} \label{driving}

It is clear that the estimates obtained in Subsection \ref{global} are not uniform in $\vep$ (note for instance that
 $Z^\varepsilon$ is of order $\varepsilon^{-1/2}$, as can be seen by computing $\mathbb E(\|Z^\epsilon(t)\|^2_{L^2})$). 
 In order to get a uniform bound in $\vep$ on the solutions of the system \eqref{systemetranslate} allowing us to prove Theorem \ref{thmfinal},
we first remark that this latter system is related to a diffusion-approximation problem.
Indeed, let us denote $(z^\varepsilon(t),\zeta^\varepsilon(t))=(\varepsilon^{1/2} Z^\varepsilon(t), \varepsilon^{3/2} \dr_t Z^\varepsilon(t))$, where we recall that $Z^\vep$ satisfies \eqref{eqZeps}. 
Then $(z^\vep,\zeta^\vep)$ satisfies the equation
\begin{equation}\label{equationzeps}
\left\{\begin{array}{ll}
\displaystyle \dr_t z^\varepsilon=\frac1\varepsilon \zeta^\varepsilon,  \\
\displaystyle d\zeta^\varepsilon=\frac1\varepsilon(-\alpha\zeta^\varepsilon+\dr_x^2{ z}^\varepsilon)dt +\frac1{\varepsilon^{1/2}}\phi d{W}_t,
\end{array}\right. 
\end{equation}
with $z^\vep(0)=\zeta^\vep(0)=0$.

Moreover, the infinitesimal generator of $(z^\vep,\zeta^\vep)$ is given by $\frac{1}{\vep} \mathcal{M}$, where $\mathcal{M}$ is the generator 
of the $\vep$-independent process $(z,\zeta)$ such that
\begin{equation}\label{equationz2}
\left\{\begin{array}{ll}
\displaystyle \dr_t z= \zeta,  \\
\displaystyle d\zeta=(-\alpha\zeta+\dr_x^2{ z})dt +\phi d{W}_t.
\end{array}\right. 
\end{equation}
Note that the transition semigroup associated with \eqref{equationz2} possesses a unique (Gaussian) invariant measure $\nu$
(see \cite[Chapter 11]{DPZ14}), where uniqueness
follows from the decay of $S_\alpha$ for $\alpha>0$. Note that $\nu$ is also invariant for the system \eqref{equationzeps}.

In the next section, we will prove a uniform bound on a modified energy by using the perturbed test function method, and this bound will allow us to 
prove the tightness of approximating sequences (as $\vep$ tends to $0$) in Section \ref{tightness}. We first gather in this section some results about 
the driving process $(z^\vep,\zeta^\vep)$ and the generator $\mathcal{M}$.

Choosing an initial condition of the system \eqref{equationz2} fo the form:
\begin{equation}\label{donnee_init_z}
\begin{pmatrix}
z_0 \\\zeta_0
\end{pmatrix}  = \int_{-\infty}^0 S_\alpha(-s)\begin{pmatrix} 0\\\phi dW_s
\end{pmatrix}
\end{equation}	  
allows to get a stationary solution of this system. 
The regularity of $(z_0,\zeta_0)$ depends on the regularity of $\phi$ as a Hilbert-Schmidt operator,  as can be seen in Section \ref{salpha}. 
Then for any continuous bounded function $\varphi$ on $H^2\times H^1\cap\dot{H}^{-1}$ we have an explicit expression for the expectation with respect to the invariant measure $\nu$: 
\begin{align*}
\mathbb{E}_\nu[\varphi(z,\zeta)]
&= \mathbb{E}\left[\varphi(z_0,\zeta_0)\right]=\mathbb{E}\left[\varphi\left( \int_{0}^\infty S_\alpha(s)\begin{pmatrix} 0\\\phi dW_s
\end{pmatrix}_1 ,   \int_{0}^\infty S_\alpha(s)\begin{pmatrix} 0\\\phi dW_s\end{pmatrix}_2\right)\right].
\end{align*}

In order to justify the computations below, we will need to take $\phi\in \mathcal{L}_2(H^3\cap\dot{H}^{-4})$,  which guarantees that $(z^\varepsilon,\zeta^\varepsilon)\in C([0,T];(H^3\cap\dot{H}^{-3})^2)$ almost surely, for any $T>0$, 
since the operator $S_\alpha(t)$ commutes with the operator $\dr_x$. We control the growth of $(z^\varepsilon,\zeta^\varepsilon)$ thanks to the following Proposition (see also \cite{BDT2023}).
\begin{prop}\label{propprocstationnaire}
For any $T>0$ and any $\delta>0$:  
$$\lim_{\vep \to 0} \PP\Big(\sup_{[0,T]}(\dbn{z^\varepsilon(t)}_{H^3\cap\dot{H}^{-3}} +\dbn{\zeta^\varepsilon(t)}_{H^3\cap\dot{H}^{-3}})\geqslant \varepsilon^{-\delta}\Big)=0.$$ 
\end{prop}
\begin{proof} Let $(z,\zeta)$ satisfy
\eqref{equationz2} with initial condition $z(0)=\zeta(0)=0$ and, for $k\in\EN$, denote by $\eta_k$ the process $$\eta_k=\sup_{t\in [k,k+1]}\left(\dbn{z(t)}_{H^3\cap\dot{H}^{-3}}+\dbn{\zeta(t)}_{H^3\cap\dot{H}^{-3}}\right).$$ 

It is easy to see that $\mathbb{E}(\eta_k^2)$ is bounded independently of $k$.
Since the process $(z,\zeta)$ is Gaussian, this implies that $\mathbb{E}[\eta_k^\gamma]$ is finite and is also bounded independtly  on $k$, for all $\gamma>0$. Thus, by Markov inequality for any $\delta'>0$, choosing $\gamma>1/\delta'$ we have $\sum_{k\in\EN}\PP(\eta_k>k^{\delta'})<\infty$, so that  Borel-Cantelli lemma implies the existence of  random variables $Z_1,Z_2$  such that
 \begin{equation}\label{borne z zeta}
\dbn{z(t)}_{H^3\cap\dot{H}^{-3}}+\dbn{\zeta(t)}_{H^3\cap\dot{H}^{-3}}\leqslant Z_1 + \vert t\vert^{\delta'} Z_2, \; \text{ a.s.}
\end{equation} 
Choosing $\delta'<\delta$, this estimate implies the statement of Proposition \ref{propprocstationnaire} for $(z(\frac{t}\varepsilon),\zeta(\frac{t}\varepsilon))$ and thus for $(z^\varepsilon,\zeta^\varepsilon)$, noticing that both processes have the same law.
\end{proof}
As a consequence of Proposition \ref{propprocstationnaire}, the stopping time 
\begin{equation}\label{deftaudelta}
\tau_\delta^\varepsilon = \inf \Big\{t: \dbn{z^\varepsilon(t)}_{H^3\cap\dot{H}^{-3}} + \dbn{\zeta^\varepsilon(t)}_{H^3\cap\dot{H}^{-3}} \geqslant \varepsilon^{-\delta}\Big\}
\end{equation}
converges to infinity in probability as $\varepsilon$ goes to $0$.
	
Let us focus on the infinitesimal generator $\MM$ of $(z,\zeta)$, which will be useful in the next sections where we use the Perturbed Test  Function method, for which we have to find the inverse of the generator applied to several quantities. 
We do not wish to specify the domain of the generator. For a Borel function $\varphi$, we write $\MM\varphi$ when $\varphi(z(t),\zeta(t))-\int_0^t\MM\varphi(z(s),\zeta(s))ds$ is a well defined quantity and 
is an integrable martingale. Similarly, we write $\MM^{-1}\varphi$ when $\MM^{-1}\varphi(z(t),\zeta(t))-\int_0^t\varphi(z(s),\zeta(s))ds$ is a well defined quantity and 
is an integrable martingale. 

For a function $f$, we try to find a function $\varphi$ such that $\MM\varphi=f$ (we write $\varphi=\MM^{-1}f$). Clearly, we need $\mathbb{E}_\nu[f(z,\zeta)]=0$, recalling that $\nu$ is the invariant measure
of $(z,\zeta)$. 
Given $\eta = (z,\zeta) $, we denote by $(z(t,\eta), \zeta(t,\eta))$ the solution of \eqref{equationz2} with initial data $\eta$ at time $0$. Then, since the process has all moments finite, it is classical that for any $\varphi$ measurable with polynomial growth, we have:
$$
\MM^{-1}\varphi(\eta)= -\mathbb E \int_0^\infty \varphi(z(t,\eta), \zeta(t,\eta))dt.
$$
The following results give properties on $\MM\varphi$ for some test functions $\varphi$ that will be used in the Perturbed Test Function method. 

\begin{lemma}\label{inverse1}
For $(z, \zeta)\in H^1\cap\dot{H}^{-2}\times L^2\cap \dot{H}^{-2}$  and $\phi\in\mathcal{L}_2(L^2,L^2\cap\dot{H}^{-2})$ we have 
\begin{equation}
\MM^{-1}z=(\dr_x^2)^{-1}\zeta +\alpha(\dr_x^2)^{-1}z
\end{equation}
where $\MM^{-1}z$ is a notation for $\MM^{-1}\pi_1(z,\zeta)$,  $\pi_1$ being the projection on the first coordinate. 
\end{lemma}
\begin{proof}
By equation \eqref{equationz2}, we have
$$d\left((\dr_x^2)^{-1}\zeta+\alpha(\dr_x^2)^{-1}z\right) = z dt +(\dr_x^2)^{-1}\phi dW_t$$ and thus $(\dr_x^2)^{-1}\zeta(t)+\alpha(\dr_x^2)^{-1}z(t) -\int_0^tz(s)ds$ is a martingale.  Note that the zero-mean condition is satisfied because $\nu$ is a centered Gaussian measure. 
\end{proof}

\begin{prop}\label{coroborneLinf}
For $(z, \zeta)\in (H^1\cap\dot{H}^{-3})\times(L^2\cap\dot{H}^{-3})$  and $\phi\in\mathcal{L}_2(L^2,H^1\cap\dot{H}^{-4})$ we have 
\begin{equation}
\dbn{\MM^{-1}\big(z\MM^{-1}z -\mathbb{E}_\nu\left[z\MM^{-1}z\right]\big)}_{L^\infty} \leqslant C_\phi\big( 1+ \dbn{z}_{H^1\cap\dot{H}^{-3}}^2 + \dbn{\zeta}_{L^2\cap\dot{H}^{-3}}^2\big)
\end{equation}
where $\mathbb{E}_\nu$ denotes the expectation under the invariant measure and $C_\phi$ is a constant depending on the Hilbert-Schmidt norm of $\phi$ in the space $H^1\cap\dot{H}^{-4}$.  

Moreover, for $(z, \zeta)\in (H^2\cap\dot{H}^{-2})\times (H^1\cap\dot{H}^{-2})$  and $\phi\in\mathcal{L}_2(L^2,H^2\cap\dot{H}^{-3})$,
inequality \eqref{coroborneLinf} holds with $(z, \zeta)$ replaced by $(\dr_x z,\dr_x \zeta)$, and a constant $C_\phi$ depending on
$\|\phi\|_{\mathcal{L}_2(L^2,H^2\cap\dot{H}^{-3})}$.
\end{prop}

In order to prove Proposition \ref{coroborneLinf}, we state the following lemma: 
\begin{lemma}\label{inverse2}
Under the assumptions of Proposition \ref{coroborneLinf},
\begin{equation}
\begin{split}
\MM^{-1}&\left(z\MM^{-1}z-\mathbb{E}_\nu\left[z\MM^{-1}z\right]\right) =-\int_0^\infty\Big[\left(S_\alpha(t)(z,\zeta)\right)_1(\dr_x^2)^{-1}\left(S_\alpha(t)(z,\zeta)\right)_2 \\
-&\sum_{k\in\EN}\int_t^\infty\left(S_\alpha(s)(0,\phi e_k)\right)_1(\dr_x^2)^{-1}\left(S_\alpha(s)(0,\phi e_k)\right)_2 ds \\
+&\alpha\left(S_\alpha(t)(z,\zeta)\right)_1(\dr_x^2)^{-1}\left(S_\alpha(t)(z,\zeta)\right)_1\\
-&\alpha\sum_{k\in\EN}\int_t^\infty\left(S_\alpha(s)(0,\phi e_k)\right)_1(\dr_x^2)^{-1}\left(S_\alpha(s)(0,\phi e_k)\right)_1 ds\Big]dt,
\end{split}
\end{equation}
where $S_\alpha$ is the semigroup associated with the linear damped wave equation (see \eqref{Salpha}, \eqref{SalphaBF}  and \eqref{SalphaHF}
for an explicit expression). Besides we denote by $(S_\alpha(t)(z,\zeta))_i$ the i-th component of $S_\alpha(t)(z,\zeta)$.
\end{lemma}
\begin{proof}
Thanks to Lemma \ref{inverse1},
\begin{align*}
\MM^{-1}\left(z\MM^{-1}z-\mathbb{E}_\nu\left[z\MM^{-1}z\right]\right)&= 
\MM^{-1}\left(z(\dr_x^2)^{-1}\zeta -\mathbb{E}_\nu\left[z(\dr_x^2)^{-1}\zeta\right]\right)\\
&+\alpha \MM^{-1}\left( z(\dr_x^2)^{-1}z -\mathbb{E}_\nu\left[ z(\dr_x^2)^{-1}z\right]\right).
\end{align*}
For the first term, we write:
\begin{equation}\label{esp_nu}
\begin{split}
\mathbb{E}_\nu[z(\dr_x^2)^{-1}\zeta]
=&\mathbb{E}\Big[\big(\int_{-\infty}^0S_\alpha(-s)(0,\phi dW_s)\big)_1\big(\int_{-\infty}^0(\dr_x^2)^{-1}S_\alpha(-s)(0,\phi dW_s)\big)_2\big] \\
=& \sum_{k\in\EN}\int_0^{\infty}\big(S_\alpha(s)(0,\phi e_k)\big)_1\big((\dr_x^2)^{-1}S_\alpha(s)(0,\phi e_k)\big)_2 ds
\end{split}
\end{equation}
thanks to It\^o isometry.  
It follows: 
\begin{align*}
\MM^{-1}\left(z(\dr_x^2)^{-1}\zeta \right.&\left. -\mathbb{E}_\nu\left[z(\dr_x^2)^{-1}\zeta\right]\right) \\=& 
 -\int_0^\infty\mathbb{E}\Big[\big(S_\alpha(t)(z,\zeta)+\int_0^tS_\alpha(t-s)(0,\phi dWs)\big)_1 \\
 \times(\dr_x^2)^{-1}&\big(S_\alpha(t)(z,\zeta)+\int_0^tS_\alpha(t-s)(0,\phi dWs)\big)_2-\mathbb{E}_\nu\left[z(\dr_x^2)^{-1}\zeta\right]\Big] dt \\
 =&-\int_0^\infty\Big[\left(S_\alpha(t)(z,\zeta)\right)_1(\dr_x^2)^{-1}\left(S_\alpha(t)(z,\zeta)\right)_2\\
 &+\sum_{k\in\EN}\int_t^\infty \left(S_\alpha(s)(0,\phi e_k)\right)_1(\dr_x^2)^{-1}\left(S_\alpha(s)(0,\phi e_k)\right)_2ds\Big] dt.
\end{align*}
The argument is the same for the second term.
\end{proof}

\begin{proof}[Proof of Proposition \ref{coroborneLinf}]
The proof is now a consequence of Lemma \ref{inverse2}, together with the embedding $H^1\hookrightarrow L^\infty$ and the bounds  
of Lemma \ref{inegalites generales}.
\end{proof}

\section{Modified energy estimate}\label{sectionenergie}

Before stating our modified energy estimate, let us explain how we apply the perturbed test function method to our problem.
Recall that  $\MM$ is the infinitesimal generator of $(z,\zeta)$. We denote by $\mathcal{L}^\varepsilon$ the infinitesimal generator of $(u,m,\mu,z^\varepsilon,\zeta^\varepsilon)$, associated with the system \eqref{systemetranslate}, that we rewrite in the form
\begin{equation} \label{systemeummu}
\left\{\begin{array}{l}
i\dr_tu=-\dr_x^2u+(m+\frac{1}{\sqrt{\vep}}z^\varepsilon)u, \\[0.2cm]
\dr_t m=\mu,\\[0.2cm]
\varepsilon^2\dr_t \mu+\alpha\varepsilon \mu=\dr_x^2(m + \vert u\vert^2),
\end{array}\right.
\end{equation}
recalling that $(z^\vep,\zeta^\vep)$ is the solution of \eqref{equationzeps} with initial condition $z^\vep(0)=\zeta^\vep(0)=0$.
Provided $\varphi$ is a smooth function, It\^o formula gives:
\begin{align*}
\mathcal{L}^\varepsilon\varphi(u,m,\varepsilon^2\mu,z^\varepsilon,  \zeta^\varepsilon)=&D_u\varphi(i\dr_x^2u-imu)-\frac{1}{\sqrt{\varepsilon}}D_u\varphi(iz^\varepsilon u) +D_m\varphi(\mu) \\&+D_\mu\varphi(\dr_x^2(m+\vert u\vert^2)-\alpha\varepsilon\mu)+\frac{1}{\varepsilon}\MM\varphi,
\end{align*}
where
\begin{equation}
\MM\varphi= D_z\varphi(\zeta)+D_\zeta\varphi(\dr_x^2z-\alpha\zeta) +\frac{1}{2}Tr(\phi^*D^2_\zeta\varphi\phi).
\end{equation}

If $\varphi$ does not depend on $(z,\zeta)$, then $\MM\varphi=0$, and in order to deal with the singular term in $\varepsilon^{-1/2}$
in the expression of $\mathcal{L}^\varepsilon \varphi$, it seems natural to add to $\varphi$ a corrector $\sqrt{\vep} \varphi_1$, with
$\MM\varphi_1=D_u\varphi(izu)$, or in other terms, according to Lemma \ref{inverse1},
\begin{equation} \label{firstcorr}
\varphi_1(u,z,\zeta)=D_u\varphi(i\MM^{-1}zu)
= D_u\varphi\left(iu(\dr_x^2)^{-1}(\zeta+\alpha z)\right).
\end{equation}
The new terms in $\mathcal{L}^\varepsilon(\varphi+\sqrt{\varepsilon}\varphi_1)$ have to be controled uniformly in $\vep$. Due to the growth 
of $(z^\varepsilon,\zeta^\varepsilon)$ as a negative power of $\varepsilon$ (see Proposition \ref{propprocstationnaire}), we are led to consider a second corrector
\begin{equation} \label{secondcorr}
\begin{split}
\varphi_2&(u,z,\zeta)=\MM^{-1}\left(D_u\varphi_1(iuz)-\mathbb{E}_\nu \left[D_u\varphi_1(iuz)\right]\right) \\
&=\MM^{-1}\left(D_u\left(D_u\varphi(iu\MM^{-1}z)\right)(iuz)-\mathbb{E}_\nu\left[ D_u\left(D_u\varphi(iu\MM^{-1}z)\right)(iuz)\right]\right),
\end{split}
\end{equation}
where $\nu$ is the invariant measure of $z^\varepsilon$, or of $z$.  Finally,  setting $\varphi^\varepsilon=\varphi+\sqrt{\varepsilon}\varphi_1+\varepsilon\varphi_2$, or more precisely,
\begin{equation}
\begin{split}
\varphi^\varepsilon(u,m,\mu,z^\varepsilon,\zeta^\varepsilon)=&\varphi(u,m,\mu)+\sqrt{\varepsilon}D_u\varphi\left(iu(\dr_x^2)^{-1}(\zeta^\varepsilon+\alpha z^\varepsilon)\right)\\&+\varepsilon\MM^{-1}\left(D_u\varphi_1(iuz^\varepsilon)-\mathbb{E}_\nu\left[D_u\varphi_1(iuz)\right]\right),
\end{split}
\end{equation}
we have a hope to bound  $\mathcal{L}^\varepsilon \varphi^\vep$ uniformly in $\vep$.

We now apply the above computations to the energy $H(u,m,\mu)$, where we recall that
$$
H(u,m,\mu)= \dbn{\dr_xu}_{L^2}^2 + \frac{1}{2}(\dbn{m}_{L^2}^2+\dbn{\varepsilon V}_{L^2}^2)+\int_\ER m\vert u\vert^2 dx,
$$ 
with $V=-\dr_x^{-1}\mu$. 
We also introduce 
\begin{equation}\label{K}
K(u,m,\mu)= \dbn{\dr_xu}_{L^2}^2 + \frac{1}{2}\dbn{m}_{L^2}^2 +\frac{1}{2}\dbn{\varepsilon V}_{L^2}^2.
\end{equation}	 
It is clear that, for some constant $C>0$,
 \begin{equation}\label{borneH}
\frac12 K(u,m,\mu) -C \|u\|_{L^2}^6  \leqslant H(u,m,\mu)  \leqslant 2K(u,m,\mu) +C\|u\|_{L^2}^6
\end{equation} 
for any $(u,m,\mu) \in H^1\times L^2 \times \dot{H}^{-1}$.

Now, the above computation of $\mathcal{L}^\vep \varphi$ applied to $\varphi=H$ (which does not depend on $(z,\zeta)$), and the fact that $H$ is 
preserved for the deterministic equation leads to
\begin{equation} \label{LepsH}
\mathcal{L}^\varepsilon H(u,m,\mu)=-\alpha\varepsilon\dbn{\dr_x^{-1}\mu}_{L^2}^2  -\frac{2}{\sqrt{\varepsilon}}   Re    \int_\ER iu\dr_x\bar{u}\dr_xz^\varepsilon  dx .
\end{equation}
The first corrector  $H_1(u,z^\vep,\zeta^\vep)$ is computed thanks to \eqref{firstcorr} and Lemma \ref{inverse1}: 
\begin{equation}\label{H1}
\begin{split}
H_1(u,z^\vep,\zeta^\vep) &=2   Re    \int_\ER iu\dr_x\bar{u}\dr_x\MM^{-1}z^\varepsilon dx\\
&= 2   Re    \int_\ER iu\dr_x\bar{u}\left(\dr_x^{-1}\zeta^\varepsilon +\alpha\dr_x^{-1}z^\varepsilon \right) dx,
\end{split}
\end{equation} 
while $H_2(u,z^\vep,\zeta^\vep)$ is given by (see \eqref{secondcorr}):
\begin{equation}\label{H2}
H_2 (u,z^\vep,\zeta^\vep)= 2\int_\ER\vert u\vert^2\MM^{-1}\left(\dr_xz^\varepsilon\dr_x\MM^{-1}z^\varepsilon-\mathbb{E}_\nu\left[\dr_xz\dr_x\MM^{-1}z\right]\right)dx.
\end{equation}

\begin{prop}\label{bornesHK}
Let  $u_0\in H^3(\ER),m_0\in H^2(\ER)\cap \dot{H}^{-1}(\ER), m_1\in H^1(\ER)\cap\dot{H}^{-1}(\ER)$, and let $(z^\vep(t),\zeta^\vep(t))$ be the unique
 solution of \eqref{equationzeps} satisfying $z^\vep(0)=\zeta^\vep(0)=0$, with $\phi \in \mathcal{L}_2(L^2, H^2\cap \dot{H}^{-3})$. Let $(u(t),m(t),\mu(t))$ be the solution of
\eqref{systemeummu} given by Theorem \ref{regular-solutions}, and consider
\begin{equation} \label{Heps}
H^\vep(t)=H(u(t),m(t),\mu(t))+\sqrt{\vep} H_1(u(t),z^\vep(t),\zeta^\vep(t)) + \vep H_2(u(t),z^\vep(t),\zeta^\vep(t)),
\end{equation}
$H_1$ and $H_2$ being respectively defined in \eqref{H1} and \eqref{H2}. Let us fix {$\delta\leqslant \frac18$}, and let $\tau^\vep_\delta$
be defined in \eqref{deftaudelta}. Then, there is a constant $C$ depending only on $\|u_0\|_{L^2}$ and 
$\|\phi\|_{\mathcal{L}_2(L^2, H^2\cap \dot{H}^{-3})}$
such that for any $\vep$ with $0<\vep\leqslant 1$,
$$
\frac14 K(u(t),m(t),\mu(t)) -C \leqslant H^\vep(t) \leqslant 3K(u(t),m(t),\mu(t)) +C, \; a.s.,
$$
for any $t<\tau^\vep_\delta$.
\end{prop}

\begin{proof}
In view of \eqref{borneH}, it suffices to obtain adequate bounds on $H_1$ and $H_2$. From
\eqref{H1}, we deduce
\begin{equation}\label{borneH_1}
\vert \sqrt{\vep} H_1(u,z^\vep,\zeta^\vep) \vert  \leqslant 2\dbn{u}_{L^2} K^{\frac12}(u,m,\mu)(\sqrt{\vep}\dbn{\dr_x^{-1}\zeta^\varepsilon}_{H^1}+\alpha\sqrt{\vep} \dbn{\dr_x^{-1}z^\varepsilon}_{H^1}), \\
\end{equation}
while \eqref{H2} and Proposition \ref{coroborneLinf} imply
\begin{equation}\label{borneH_2}
\begin{split}
\vert \vep H_2(u,z^\vep,\zeta^\vep)\vert \leqslant 2\vep \dbn{u}_{L^2}^2\dbn{\MM^{-1}\left(\dr_xz^\varepsilon\dr_x\MM^{-1}z^\varepsilon-\mathbb{E}_\nu\left[\dr_xz\dr_x\MM^{-1}z\right]\right)}_{L^\infty} \\
\leqslant C_\phi \|u\|_{L^2}^2 (1+ \vep \dbn{z^\varepsilon}_{H^2\cap\dot{H}^{-2}}^2+ \vep \dbn{ \zeta^\varepsilon}_{H^1\cap\dot{H}^{-2}}^2).
\end{split}
\end{equation}
The conclusion follows thanks to Young's inequality, \eqref{deftaudelta} and the conservation of $L^2$ norm.
\end{proof}

Next, we need to compute the evolution of $H^\vep(t)$ defined in \eqref{Heps}, in order to obtain a uniform estimate with respect to $\vep$ 
on the solution. The It\^o formula will give us an evolution of the form 
\begin{equation} \label{evolHeps}
dH^\vep(t) =\mathcal{L}^\vep H^\vep(t) +dX_t+dY_t,
\end{equation}
where $X_t$ and $Y_t$ are martingale terms arising from the correctors $H_1$ and $H_2$. The estimate is as follows.

\begin{prop}\label{boundpolynome}
Under the assumptions of Proposition \ref{bornesHK}, there exist positive constants $B$ and $C$, depending only on $\|u_0\|_{L^2}$
and $\|\phi\|_{\mathcal{L}_2(L^2,H^3\cap\dot{H}^{-3})}$ such that
for any $\varepsilon$ with $0<\varepsilon\leqslant1$, 
\begin{equation}
\mathcal{L}^\varepsilon(H^\varepsilon(t)) +\alpha\dbn{\sqrt{\varepsilon}V(t)}_{L^2}^2\leqslant \varepsilon K_t^2+BK_t+C, \; a.s.,
\end{equation}
for any $t\in  [0,\tau^\varepsilon_\delta]$, where $K_t=K(u(t),m(t),\mu(t))$, and $V(t)=-\dr_x^{-1}\mu(t)$.
\end{prop}

\begin{proof}
By \eqref{LepsH}, and It\^o formula applied to \eqref{H1} using \eqref{equationzeps} and \eqref{systemeummu}, we easily obtain
\begin{equation}\label{H+sqrtepsH1}
\begin{split}
d(H+\sqrt{\varepsilon}H_1)=&\Big(-\alpha\varepsilon\dbn{\dr_x^{-1}\mu}_{L^2}^2-2\int_\ER\vert u\vert^2\dr_xz^\varepsilon (\dr_x^{-1}\zeta^\varepsilon +\alpha\dr_x^{-1}z^\varepsilon ) dx \\
&+ 4\sqrt{\varepsilon}\int_\ER \vert \dr_x u\vert ^2(\zeta^\varepsilon +\alpha z^\varepsilon ) dx + \sqrt{\varepsilon}\int_\ER \dr_x\vert u\vert^2(\dr_x\zeta^\varepsilon +\alpha\dr_x z^\varepsilon ) dx  \\&-2\sqrt{\varepsilon}\int_\ER\vert u\vert^2\dr_xm (\dr_x^{-1}\zeta^\varepsilon +\alpha\dr_x^{-1}z^\varepsilon ) dx\Big)dt \\
&+2   Re    \int_\ER iu\dr_x\bar{u}\sum_{k\in\EN}\dr_x^{-1}(\phi e_k) d\beta_k(t)dx \\
=& \mathcal{L}^\varepsilon (H+\sqrt{\varepsilon}H_1) dt + dX_t,
\end{split} 
\end{equation}
with\begin{equation}\label{X_t}
\begin{split}
X_t&= 2   Re    \int_\ER\int_0^t iu\dr_x\bar{u}\sum_{k\in\EN}\dr_x^{-1}(\phi e_k) d\beta_k(s)dx  .
\end{split}
\end{equation}

Introducing then the corrector $H_2$ defined in \eqref{H2} in order to control the terms of order zero in $\vep$ in the above expression
of $\mathcal{L}^\varepsilon (H+\sqrt{\varepsilon}H_1)$, we obtain
\begin{equation}\label{generateurcorrige}
\begin{split}
\mathcal{L}^\varepsilon(H^\varepsilon_t ) =& -2\int_\ER\vert u\vert^2\mathbb{E}_\nu\left[\dr_xz \dr_x\MM^{-1}z \right]dx \\
+& \sqrt{\varepsilon}\Big(4\int_\ER\vert \dr_xu\vert^2(\zeta^\varepsilon +\alpha z^\varepsilon )dx +\int_\ER\dr_x\vert u\vert^2(\dr_x\zeta^\varepsilon +\alpha \dr_xz^\varepsilon )dx \\&-2\int_\ER\vert u\vert^2\dr_xm(\dr_x^{-1}\zeta^\varepsilon +\alpha\dr_x^{-1}z^\varepsilon )dx\Big) \\
-&\alpha\dbn{\sqrt{\varepsilon}V}_{L^2}^2+4\varepsilon    Re    \int_\ER i\bar{u}\dr_x^2u\MM^{-1}\left(\dr_xz^\varepsilon \dr_x\MM^{-1}z^\varepsilon -\mathbb{E}_\nu\left[\dr_xz \dr_x\MM^{-1}z \right]\right)dx,
\end{split}
\end{equation}
where we recall that $H^\vep_t$ is defined in \eqref{Heps} and $V=-\dr_x^{-1}\mu$.
Note that the second corrector also contributes to the martingale part which can be computed thanks to It\^o formula and Lemma \ref{inverse2} adding
\begin{equation}\label{Y_t}
\begin{split}
Y_t=&{\sqrt{\varepsilon}}\int_{\!s=0}^t\int_\ER\!\vert u\vert^2\int_{\!t'=0}^\infty \!\!\dr_x\!\left(S_\alpha(t')(z^\varepsilon, \phi dW_s)\right)_1 
\dr_x^{-1}\!\left(\left(S_\alpha(t')(z^\varepsilon, \zeta^\varepsilon)\right)_2+\alpha \!\left(S_\alpha(t')(z^\varepsilon, \zeta^\varepsilon)\right)_1\right) \\
&+\dr_x\!\left(S_\alpha(t')(z^\varepsilon, \zeta^\varepsilon)\right)_1 \dr_x^{-1}\!\left(\left(S_\alpha(t')(z^\varepsilon, \phi dW_s)\right)_2+\alpha\!\left(S_\alpha(t')(z^\varepsilon, \phi dW_s)\right)_1\right)dtdx .
\end{split} 
\end{equation}
In other words,  $dH^\varepsilon_t=\mathcal{L}^\varepsilon H^\varepsilon_t dt + d(X_t+Y_t)$ with $X_t$ defined in \eqref{X_t} and $Y_t$ in \eqref{Y_t}.

It remains to bound the terms in the right hand side of  \eqref{generateurcorrige}.
First, Lemma \ref{inverse1}, a computation similar to \eqref{esp_nu}, and Lemma \ref{inegalites generales} easily gives
 \begin{align*}
\left\vert 2\int_\ER \vert u\vert^2\mathbb{E}_\nu\left[\dr_xz\dr_x\MM^{-1}z\right]dx \right\vert\leqslant C \|\phi\|_{\mathcal{L}(L^2;H^1\cap \dot{H}^{-2})}^2 \dbn{u}_{L^2}^2.
\end{align*}
Besides, for $t \in [0,\tau^\varepsilon_\delta]$,
 \begin{align*}
\left\vert 4\sqrt{\varepsilon}\int_\ER\vert\dr_xu\vert^2(\zeta^\varepsilon+\alpha z^\varepsilon)dx\right\vert &\leqslant C\sqrt{\varepsilon}\dbn{\dr_xu}_{L^2}^2\left(\dbn{\zeta^\varepsilon}_{H^1}+\alpha\dbn{z^\varepsilon}_{H^1}\right) \\
&\leqslant C\varepsilon^{\frac{1}{2}-\delta}K_t,
\end{align*}
and
\begin{align*}
\left\vert\sqrt{\varepsilon}\int_\ER\dr_x\vert u\vert^2\left(\dr_x\zeta^\varepsilon+\alpha\dr_xz^\varepsilon\right)dx \right\vert 
&\leqslant C \sqrt{\varepsilon}\dbn{u}_{L^2}\dbn{\dr_xu}_{L^2}
\left(\dbn{\zeta^\varepsilon}_{H^2}+\alpha\dbn{z^\varepsilon}_{H^2}\right) \\
&\leqslant C\varepsilon^{\frac{1}{2}-\delta}K_t^\frac{1}{2} 
\leqslant C+K_t,
\end{align*}
thanks to the conservation of $\|u\|_{L^2}$ and since $\vep\leqslant 1$.
It remains to estimate two terms in the right-hand side of \eqref{generateurcorrige} . First, by integration by parts: 
\begin{align*}
\sqrt{\varepsilon}\left\vert\int_\ER\vert u\vert^2\dr_xm\left(\dr_x^{-1}\zeta^\varepsilon+\alpha\dr_x^{-1}z^\varepsilon\right)dx\right\vert \leqslant& \sqrt{\varepsilon}\int_\ER\dr_x\vert u\vert^2m\left(\dr_x^{-1}\zeta^\varepsilon+\alpha\dr_x^{-1}z^\varepsilon\right)dx \\&+ \sqrt{\varepsilon}\int_\ER\vert u\vert^2m\left(\zeta^\varepsilon+\alpha z^\varepsilon\right)dx.
\end{align*}
The second integral is easily bounded by $C+K_t$, for $t\in [0,\tau^\varepsilon_\delta]$, using similar arguments as above and in those in Lemma \ref{GronwallH},
while for the first integral: 
\begin{align*}
\sqrt{\varepsilon}\bigg\vert\int_\ER\dr_x\vert u\vert^2m&\left(\dr_x^{-1}\zeta^\varepsilon+\alpha\dr_x^{-1}z^\varepsilon\right)dx\bigg\vert \\&\leqslant 2\sqrt{\varepsilon}\left(\dbn{\dr_x^{-1}\zeta^\varepsilon}_{L^\infty}+\alpha\dbn{\dr_x^{-1}z^\varepsilon}_{L^\infty}\right)\dbn{u}_{L^\infty}\dbn{\dr_xu}_{L^2}\dbn{m}_{L^2} \\
&\leqslant C\dbn{u}_{L^2}^\frac{1}{2} \varepsilon^{\frac{1}{2}-\delta}\dbn{\dr_xu}_{L^2}^{\frac32}\dbn{m}_{L^2} 
{\leqslant C (K_t+\vep^{1-2\delta} K_t^{\frac32})} \\
& \leqslant{ C(1+K_t+\vep K_t^2),}
\end{align*}
again thanks to the conservation of $\|u\|_{L^2}$, and for $t\in [0,\tau^\varepsilon_\delta]$, recalling that $\delta \leqslant \frac18$.
The last term in \eqref{generateurcorrige} is also integrated by parts, and taking account of the fact that $z^\vep$ and $\zeta^\vep$
are real valued, it is bounded for $t \in [0,\tau^\varepsilon_\delta]$ by
\begin{align*}
 C\varepsilon &\dbn{u}_{L^2}\dbn{\dr_xu}_{L^2}\dbn{\dr_x\MM^{-1}\left(\dr_xz^\varepsilon\dr_x\MM^{-1}z^\varepsilon-\mathbb{E}_\nu\left[\dr_xz\dr_x\MM^{-1}z\right]\right)}_{L^\infty} \\
&\leqslant C_\phi \varepsilon K_t^\frac{1}{2}(1 + \dbn{z^\varepsilon}_{H^3\cap\dot{H}^{-2}}^2 +\dbn{\zeta^\varepsilon}_{H^2\cap\dot{H}^{-2}}^2) 
\leqslant C\vep^{1-2\delta} K_t^{\frac12}\\
&\leqslant C(1+K_t),
\end{align*}
thanks to Proposition \ref{coroborneLinf}, and the conservation of $\|u\|_{L^2}$.
Gathering all these estimates gives the conclusion.
\end{proof}
Propositions \ref{bornesHK} and \ref{boundpolynome}, together with the expression of the martingale terms \eqref{X_t} and \eqref{Y_t} 
allow us to state the two following energy estimates	which will be useful for the tightness.

\begin{prop} \label{borneespK}
Let  $u_0\in H^3(\ER),m_0\in H^2(\ER)\cap \dot{H}^{-1}(\ER), m_1\in H^1(\ER)\cap\dot{H}^{-1}(\ER)$, and let $(z^\vep(t),\zeta^\vep(t))$ be the unique
 solution of \eqref{equationzeps} satisfying $z^\vep(0)=\zeta^\vep(0)=0$, with $\phi \in \mathcal{L}_2(L^2, H^3\cap \dot{H}^{-3})$. Then, for any $T>0$,   there exists a constant 
$C(T)>0$, independent of $\varepsilon$, and a stopping time $\tau^\varepsilon$ such that $\mathbb P(\tau^\varepsilon \le T)$ converges to $0$ as $\vep$ tends to $0$, such that the 
solution $(u,m, \mu)$ of the system \eqref{systemeummu} given by Theorem  \ref{regular-solutions} satisfies 
$$\mathbb{E}\big[\sup_{t\leqslant \tau^\varepsilon\wedge T} K^2(t)\big] \leqslant C(T),$$
where $K(t)$ is defined in equation \eqref{K} with $u=u(t),m=m(t),\mu=\dr_tm(t)$. 
Moreover, if $V(t)=-\dr_x^{-1} \mu(t)$, then
$$\mathbb{E}\left[\dbn{\sqrt{\varepsilon}V}^2_{L^2(0,\tau^\varepsilon\wedge T;L^2)}\right] \leqslant C(T).$$ 
\end{prop}

The proof of Proposition \ref{borneespK} is postponed to the Appendix.

\section{Tightness of the processes}\label{tightness}
		
Let us fix $T>0$. Thanks to the bounds obtained in Proposition \ref{borneespK},  we are able to prove the tightness of the family
of processes indexed by $\vep$, and apply the Prokhorov and Skorohod Theorems.
		
Let $(u_0,m_0,m_1)$
and $(z^\vep,\zeta^\vep)$ be as in Proposition \ref{borneespK}. 
For $\varepsilon>0$, denote by $(u^\varepsilon,m^\varepsilon, \mu^\vep)$ the solution to the system \eqref{systemeummu} given by 
Theorem \ref{regular-solutions}. 
				We will use Aldous criterion (see \cite[Theorem 16.10]{Bil99}) for which we need a control of the modulus of continuity.  It leads us to consider the following integrals (we recall that $V^\varepsilon=-\dr_x^{-1}\mu^\varepsilon$):
		\begin{equation}\label{Mepsi_Vepsi}
		M^\varepsilon(t)=\int_0^tm^\varepsilon(s)ds, \qquad\qquad \mathcal{V}^\varepsilon(t)=\int_0^t V^\varepsilon(s)ds.
		\end{equation}
  		We also fix $\delta\leqslant\frac{1}{8}$ and introduce the space \begin{equation}\label{defE}
		E=\left\{f\in C(\ER,H^3\cap\dot{H} ^{-3}), \quad \sup_{t\in \ER}\frac{\dbn{f(t)}_{H^3\cap\dot{H}^{-3}}}{1+\vert t\vert^\delta}< \infty\right\}.
		\end{equation}
According to Proposition \ref{propprocstationnaire}, and more precisely to \eqref{borne z zeta}, the process $(z,\zeta)$, solution of 
\eqref{equationz2},  takes values a.s. in $E\times E$.   
Our aim is to prove the following Proposition.

\begin{prop}\label{tensionuplet}
		Let $(u_0,m_0,m_1)\in H^3\times(H^2\cap\dot{H}^{-1})\times(H^1\cap\dot{H}^{-1})$. 
		Denote by $(u^\varepsilon,m^\varepsilon,\mu^\vep)$ the solution of \eqref{systemeummu}, and $(z,\zeta)$ the solution of \eqref{equationz2}
		satisfying $z(0)=\zeta(0)=0$. Then the sequence $(u^\varepsilon,M^\varepsilon, \sqrt{\varepsilon}\mathcal{V}^\varepsilon,z,\zeta)$ where $M^\varepsilon,\mathcal{V}^\varepsilon$ are defined in \eqref{Mepsi_Vepsi} is tight in the space $\mathcal{C}^0([0,T],H^s_{loc})\times  \mathcal{C}^0([0,T],H^{-\sigma}_{loc})\times H^s([0,T],H^{-\sigma}_{loc})\times E\times E$ for any $s<1,\sigma>0$.
		\end{prop}
		In order to prove this result, we first study the process $u^\varepsilon$, for which we will apply Aldous criterion,
using again the Perturbed Test function method.
More precisely, we show that the process $(u^\varepsilon)_\varepsilon$ is tight in the space $\mathcal{C}^0([0,T],H^s_{loc})$.  
		 We need two statements, namely Lemma \ref{hypothese1} and Proposition \ref{hypothese2}.

\begin{lemma}\label{hypothese1}
Let $\varepsilon>0$ and denote by $(u^\varepsilon,m^\varepsilon,\mu^\vep)$ the solution of \eqref{systemeummu} given by 
Theorem \ref{regular-solutions}.  Then 
 \begin{equation}\label{condAldous1}
\lim_{R\to\infty}\limsup_{\varepsilon\to 0}\mathbb{P}\Big(\sup_{t\in[0,T]}\dbn{u^\varepsilon(t)}_{H^1}>R\Big)=0.
\end{equation}
\end{lemma}
\begin{proof}
This is a simple consequence of Proposition \ref{borneespK}, since
\begin{equation}
\begin{split}
\PP\Big(\sup_{t\in[0,T]}\dbn{u^\varepsilon(t)}_{H^1}^2>R\Big) \leqslant& \PP\Big(\sup_{t\in[0,\tau^\varepsilon]}\dbn{u^\varepsilon(t)}_{H^1}^2>R\Big) + \PP\left(\tau^\varepsilon <T\right) \\
&\leqslant \frac{C(T)}{R} + \PP\left(\tau^\varepsilon <T\right),
\end{split}
\end{equation} 
for some constant $C$ independent of $\vep$, and $\tau^\vep
$ converges to $T$ in probability as $\vep$ tends to $0$.
\end{proof}

\begin{prop}\label{hypothese2}
Let $\varepsilon>0$ and denote by $(u^\varepsilon,m^\varepsilon,\mu^\vep)$ the solution of \eqref{systemeummu} given by Theorem \ref{regular-solutions}.  Then for every $\lambda,\eta$, there exist $\delta_0,\varepsilon_0$ such that for $\bar{\delta}<\delta_0$ and $\varepsilon<\varepsilon_0$, if $\tau$ is a stopping time with $\tau \leqslant T$ a.s.  then 
\begin{equation}\label{condAldous2}
\PP\big(\dbn{u^\varepsilon(\tau+\bar{\delta})-u^\varepsilon(\tau)}_{H^{-\frac{1}{2}}}>\lambda\big)\leqslant\eta.
\end{equation}
\end{prop}
In order to prove Proposition \ref{hypothese2}, we will need two lemmas, whose proofs will use the Perturbed Test Function method. 

\begin{lemma}\label{lemmaAldous1}
Let $\tau^\vep$ be as in Proposition \ref{borneespK}.
For any $ \eta>0$ ,  there exist $\delta_0,\varepsilon_0$ such that for any bounded stopping time $\tau$,  for $\bar{\delta}<\delta_0$ and $\varepsilon<\varepsilon_0$ we have
\begin{equation}
\mathbb{E}\Big[\dbn{u^\varepsilon((\tau+\bar{\delta})\wedge\tau^\varepsilon)}_{H^{-\frac{1}{2}}}^2-\dbn{u^\varepsilon(\tau\wedge\tau^\varepsilon)}_{H^{-\frac{1}{2}}}^2\Big] \leqslant\eta.
\end{equation}
\end{lemma}

\begin{proof}
We apply the perturbed test function method to $\varphi(u)=\dbn{u}_{H^{-\frac{1}{2}}}^2$, for $u\in H^3$.  We first compute the infinitesimal generator $\mathcal{L}^\varepsilon$ applied to $\varphi$ and we get 
\begin{equation}
\mathcal{L}^\varepsilon\varphi(u,m,z^\vep)=2\langle (1-\dr_x^2)^{-\frac{1}{2}}u, i\dr_x^2u-im u\rangle + \frac{2}{\sqrt{\varepsilon}}\langle (1-\dr_x^2)^{-\frac{1}{2}}u, -iz^\varepsilon u\rangle.
\end{equation}
As above, in order to cancel the term of order $\vep^{-1/2}$, we use a first corrector 
\begin{equation} \label{phi1}
\begin{split}
\varphi_1(u,z^\vep,\zeta^\vep)&=2\langle (1-\dr_x^2)^{-\frac{1}{2}}u, iu\MM^{-1}z^\varepsilon \rangle, 
\end{split}
\end{equation}
which is easily bounded on $[0,\tau^\vep_\delta]$ by $C(T) \vep^{-\delta}$, thanks to Lemma \ref{inverse1}, Lemma \ref{deftaudelta}, and
the conservation of $\|u^\vep(t)\|_{L^2}$.
Applying then the infinitesimal generator $\mathcal{L}^\varepsilon$ to $\varphi(u)+\sqrt{\varepsilon}\varphi_1$ we get
\begin{align*}
\mathcal{L}^\varepsilon(\varphi+\sqrt{\varepsilon}\varphi_1)&(u,m,z^\vep,\zeta^\vep)= 2\langle (1-\dr_x^2)^{-\frac{1}{2}}u, i\dr_x^2u-imu\rangle \\
&+ 2\langle(1-\dr_x^2)^{-\frac{1}{2}}u,uz^\varepsilon \MM^{-1}z^\varepsilon \rangle - 2\langle (1-\dr_x^2)^{-\frac{1}{2}}(z^\varepsilon u),u\MM^{-1}z^\varepsilon \rangle  \\
&+2\sqrt{\varepsilon}\left[\langle(1-\dr_x^2)^{-\frac{1}{2}}u,(m^\varepsilon u-\dr_x^2u)\MM^{-1}z^\varepsilon \rangle \right.\\&\left.+\langle (1-\dr_x^2)^{-\frac{1}{2}}(\dr_x^2u-m^\varepsilon u), u\MM^{-1}z^\varepsilon \rangle\right].
\end{align*}
Here again, we introduce a second corrector $\varphi_2$ in order to control the terms of order $0$:
\begin{align*}
\varphi_2(u,z^\vep,\zeta^\vep) &= -2\langle(1-\dr_x^2)^{-\frac{1}{2}}u, u\MM^{-1}(z^\varepsilon \MM^{-1}z^\varepsilon -\mathbb{E}_\nu\left[z^\varepsilon \MM^{-1}z^\varepsilon \right])\rangle \\
+& 2\langle u,\MM^{-1}[(1-\dr_x^2)^{-\frac{1}{2}}(z^\varepsilon u)\MM^{-1}z^\varepsilon-\mathbb{E}_\nu\big[(1-\dr_x^2)^{-\frac{1}{2}}(z u)\MM^{-1}z \big])\rangle. 
\end{align*}
The two terms on the right-hand side of the above  equation are bounded for $t\in[0,\tau^\varepsilon_\delta]$ thanks to Proposition \ref{coroborneLinf} 
(or similar arguments as in the proof of Proposition \ref{coroborneLinf} for the second term), \eqref{deftaudelta} and the conservation of 
$\|u^\vep(t)\|_{L^2}$, and we get
\begin{equation} \label{estphi2}
\varphi_2(u^\vep,z^\vep,\zeta^\vep) \leqslant C(T)\varepsilon^{-2\delta}
\end{equation}
where $C$ depends on the $L^2$ norm of $u$ and the Hilbert-Schmidt norm of $\phi$.
Defining $\varphi^\varepsilon=\varphi+\sqrt{\varepsilon}\varphi_1+\varepsilon\varphi_2$,  we end up thanks to \eqref{phi1}, Lemma \ref{inverse1}
and \eqref{estphi2} with
\begin{equation} \label{estphieps}
\vert\varphi^\varepsilon(u^\varepsilon_t,z^\vep_t,\zeta^\vep_t)-\varphi(u^\varepsilon_t)\vert \leqslant C\varepsilon^{\frac{1}{2}-\delta}+C\varepsilon^{1-2\delta}, \text{ for all } t\in [0,\tau_\delta^\vep],
\end{equation} 
where $C$ depends on $T$,  $\|u_0\|_{L^2}$ and $\|\phi\|_{\mathcal{L}_2(L^2,H^3\cap\dot{H}^{-4})}$. 
We also compute
\begin{align*}
(\mathcal{L}^\varepsilon\varphi^\varepsilon)(u^\varepsilon,m^\vep,&z^\vep,\zeta^\vep) =  2\langle (1-\dr_x^2)^{-\frac{1}{2}}u^\varepsilon, i\dr_x^2u^\varepsilon-im^\varepsilon u^\varepsilon\rangle \\
&+ 2\mathbb{E}_\nu\left[\langle(1-\dr_x^2)^{-\frac{1}{2}}u^\varepsilon,u^\varepsilon z \MM^{-1}z \rangle - 2\langle (1-\dr_x^2)^{-\frac{1}{2}}(z^\varepsilon u^\varepsilon),u^\varepsilon\MM^{-1}z\rangle\right]  \\
&+2\sqrt{\varepsilon}\left[\langle(1-\dr_x^2)^{-\frac{1}{2}}u^\varepsilon,(m^\varepsilon u^\varepsilon-\dr_x^2u^\varepsilon)\MM^{-1}z^\varepsilon \rangle \right.\\&\left.+\langle i(1-\dr_x^2)^{-\frac{1}{2}}(\dr_x^2u^\varepsilon-m^\varepsilon u^\varepsilon), iu^\varepsilon\MM^{-1}z^\varepsilon \rangle\right] \\
&+\sqrt{\varepsilon} D_u\varphi_2(-iz^\varepsilon u^\varepsilon) +\varepsilon D_u\varphi_2(i\dr_x^2u^\varepsilon-im^\varepsilon u^\varepsilon).
\end{align*}
Proceeding as above, it is not difficult to prove that for $t\in [0,\tau^\vep_\delta]$,
\begin{align*}
\sqrt{\varepsilon}\vert D_u\varphi_2.(-iz^\varepsilon u^\varepsilon)\vert &
\leqslant C(T) \varepsilon^{\frac{1}{2}-3\delta},
\end{align*}
and
\begin{align*}
\varepsilon D_u\varphi_2.(i\dr_x^2u^\varepsilon-im^\varepsilon u^\varepsilon) 
& \leqslant C(T) (\dbn{u^\varepsilon}_{H^1}^2 + \dbn{m^\varepsilon }_{L^2}^2) \varepsilon^{1-2\delta},
\end{align*}
so that if $\tau^\vep$ is as in Proposition \ref{borneespK},
$$ \vert (\mathcal{L}^\varepsilon\varphi^\varepsilon)(u^\varepsilon,m^\vep,z^\vep,\zeta^\vep)\vert 
\leqslant  C(T)
\left( 1+\varepsilon^{\frac{1}{2}-3\delta}+\varepsilon^{1-2\delta}\right).
$$
Now, since  
$$\varphi^\varepsilon(u^\varepsilon_t,z^\vep_t,\zeta^\vep_t)-\varphi^\varepsilon(u_0,z^\vep_0,\zeta^\vep_0)
-\int_0^t(\mathcal{L}^\varepsilon\varphi^\varepsilon)(u^\varepsilon_s,m^\vep_s,z^\vep_s,\zeta^\vep_s)ds$$ 
is a martingale, the above estimate and \eqref{estphieps}
lead to
\begin{align*}
\mathbb{E}\big[\varphi&(u^\varepsilon((\tau+\bar{\delta})\wedge\tau^\varepsilon))-\varphi(u^\varepsilon(\tau\wedge\tau^\varepsilon))\big] \\
\leqslant &\;  \mathbb{E}\left[\varphi^\varepsilon(u^\varepsilon((\tau+\bar{\delta})\wedge\tau^\varepsilon))- \varphi^\varepsilon(u^\varepsilon(\tau\wedge\tau^\varepsilon))\right]
+C\varepsilon^{\frac{1}{2}-\delta} + C\varepsilon^{1-2\delta} \\
\leqslant & \; C \bar{\delta}\left( 1+\varepsilon^{\frac{1}{2}-3\delta}+\varepsilon^{1-2\delta}\right)+C\varepsilon^{\frac{1}{2}-\delta}+ C\varepsilon^{1-2\delta}.
\end{align*}
Finally, for $\bar{\delta}$ and $\varepsilon$ small enough we get the conclusion of Lemma \ref{lemmaAldous1}, recalling that 
$\varphi(u)=\dbn{u}_{H^{-\frac{1}{2}}}^2$.
\end{proof}

The second lemma states a similar result.
\begin{lemma}\label{lemmaaldous2}
Let $\tau^\vep$ be as in Proposition \ref{borneespK}. There exist $\delta_0,\varepsilon_0$ such that for any bounded stopping time $\tau$,  for $\bar{\delta}<\delta_0$ and $\varepsilon<\varepsilon_0$,
\begin{equation}
\mathbb{E}\left[\langle u^\varepsilon((\tau+\bar{\delta})\wedge\tau^\varepsilon)-u^\varepsilon(\tau\wedge\tau^\varepsilon), u^\varepsilon(\tau\wedge\tau^\varepsilon)\rangle_{H^{-\frac{1}{2}}}\right] \leqslant \eta
\end{equation}
\end{lemma}
\begin{proof}

The proof is very similar to the proof of Lemma \ref{lemmaAldous1}, by applying  the Pertubed Test Function method to $\varphi(u)=\langle u,h\rangle_{H^{-\frac{1}{2}}}$ for a fixed function $h\in H^{-\frac{1}{2}}$, then choosing $h=u^\varepsilon(\tau\wedge\tau^\varepsilon)$, so we
leave the details to the reader.
\end{proof}

Proposition \ref{hypothese2} is now a consequence of  Lemmas \ref{lemmaAldous1} and \ref{lemmaaldous2}.

\begin{proof}[Proof of Proposition \ref{hypothese2}]
Let $\lambda$ and $\eta$ be two positive numbers, and let $\tau$ be a stopping time such that $\tau<T-\bar \delta$. Then: 
\begin{align*}
\PP\left(\dbn{u^\varepsilon(\tau+\bar{\delta})-u^\varepsilon(\tau)}_{H^{-\frac{1}{2}}}>\lambda\right) 
\leqslant& \PP\left(\dbn{u^\varepsilon((\tau+\bar{\delta})\wedge\tau^\varepsilon)-u^\varepsilon(\tau\wedge\tau^\varepsilon)}_{H^{-\frac{1}{2}}}>\lambda\right)\\
&+\PP(\tau^\varepsilon<T) \\
\leqslant& \PP\left(\dbn{u^\varepsilon((\tau+\bar{\delta})\wedge\tau^\varepsilon)-u^\varepsilon(\tau\wedge\tau^\varepsilon)}_{H^{-\frac{1}{2}}}>\lambda\right)+ \frac{\eta}{2}
\end{align*}
for $\varepsilon$ small enough. Now, since 
\begin{align*}
\mathbb{E}&\left[\dbn{u^\varepsilon((\tau+\bar{\delta})\wedge\tau^\varepsilon)-u^\varepsilon(\tau\wedge\tau^\varepsilon)}^2_{H^{-\frac{1}{2}}}\right] \\&= \mathbb{E}\left[\dbn{u^\varepsilon((\tau+\bar{\delta})\wedge\tau^\varepsilon)}_{H^{-\frac{1}{2}}}^2-\dbn{u^\varepsilon(\tau\wedge\tau^\varepsilon)}^2_{H^{-\frac{1}{2}}}\right] \\
&-2\mathbb{E}\left[\langle u^\varepsilon((\tau+\bar{\delta})\wedge\tau^\varepsilon)-u^\varepsilon(\tau\wedge\tau^\varepsilon),u^\varepsilon(\tau\wedge\tau^\varepsilon)\rangle_{H^{-\frac{1}{2}}}\right],
\end{align*}
the conclusion follows from Markov inequality
and  Lemmas \ref{lemmaAldous1} and \ref{lemmaaldous2}, for $\varepsilon$  and $\bar{\delta}$ small enough.
\end{proof}

The next proposition is now a consequence of Lemma \ref{hypothese1} and Proposition \ref{hypothese2}, thanks to Aldous criterion and
the compactness of the embedding of $H^1$ into $H^s_{loc}$ for $s<1$.
\begin{prop}
The sequence $(u^\varepsilon)$ is tight in $\mathcal{C}^0([0,T],H^s_{loc})$ for $s<1$.
\end{prop}

Now that we have obtained the tightness of $(u^\varepsilon)_\varepsilon$, let us briefly explain how to get the tightness of the whole sequence 
$(u^\varepsilon,M^\varepsilon, \sqrt{\varepsilon}\mathcal{V}^\varepsilon,z,\zeta)$.
For the sequence $(M^\varepsilon)_\varepsilon$ defined in \eqref{Mepsi_Vepsi} we use again Aldous criterion and Proposition \ref{borneespK} coupled with the compact embedding $L^2\hookrightarrow H^{-\sigma}_{loc}$ to get the tightness in $\mathcal{C}^0([0,T],H^{-\sigma}_{loc})$.
For the sequence $(\sqrt{\varepsilon}\mathcal{V}^\varepsilon)_\varepsilon$ defined in \eqref{Mepsi_Vepsi}, we also use Proposition  
\ref{borneespK} and the compact embedding $H^1([0,T],L^2)\hookrightarrow H^s([0,T], H_{loc}^{-\sigma})$ to conclude. {Note that in these 
cases, we do not need to use the perturbed test function method to estimate the modulus of continuity, since we have taken integrals in time, so that
the estimates of Proposition \ref{borneespK} are sufficient.}
Finally,  $z$ and $\zeta$ are constant random processes, which both belong almost surely to $E$ defined in \eqref{defE}, according to the proof of Proposition \ref{propprocstationnaire} thus the tightness is established.

	\section{Convergence to the Stochastic Schr\"odinger equation }
	
	In this  section, we prove the main result of this paper, that is Theorem \ref{thmfinal}.  With this aim in view, we first prove some weak convergence results, similar to those we can find in \cite{AdAd88}, before passing to the limit in $\varepsilon$  in the martingale problem.  
	The regularities on the initial data and on $\phi$ assumed in Theorem \ref{thmfinal} ensure that the process $(z,\zeta)\in H^3\cap \dot{H}^{-3}$ and that the computations below are justified. 
		\subsection{Weak convergences}
		\label{Weak convergences}
		We have proved in the previous section that the sequence  $(u^\varepsilon,M^\varepsilon,\sqrt{\varepsilon}\mathcal{V}^\varepsilon,z,\zeta)$ is tight in $C([0,T],H^s_{loc})\times C([0,T],H^{-\sigma}_{loc})\times H^s([0,T],H^{-\sigma}_{loc})\times E\times E$ for $s<1$ and $\sigma<0$. Then, according to Skorohod Theorem, there exists a probability space $(\tilde{\Omega},\tilde{\mathcal{F}},\tilde{\PP})$ and a sequence of random variables $(\tilde{u}^\varepsilon,\tilde{M}^\varepsilon,\tilde{\mathcal{V}}^\varepsilon,\tilde{z}^\varepsilon,\tilde{\zeta}^\varepsilon)$ on $\tilde{\Omega}$, equal in law to  $(u^\varepsilon,M^\varepsilon,\sqrt{\varepsilon}\mathcal{V}^\varepsilon,z,\zeta)$, which converges almost surely in the same space to $(\tilde{u},\tilde{M},\tilde{\mathcal{V}},\tilde{z},\tilde{\zeta})$ as $\varepsilon$ goes to $0$. 
Let us define
\begin{equation}\label{defderiveesSko}
\tilde{m}^\varepsilon=\dr_t\tilde{M}^\varepsilon, \quad \tilde{m}= \dr_t\tilde{M}, \quad \tilde{V}^\varepsilon=\frac{1}{\sqrt{\varepsilon}}\dr_t\tilde{\mathcal{V}}^\varepsilon, \quad \text{and} \quad \tilde V=\dr_t \tilde{ \mathcal{V}},
\end{equation}
where the derivatives are taken in the weak sense.
We also define the process 
\begin{equation}\label{tildeZeps}
\left(\tilde{\z}^\varepsilon(t),\tilde{\zz}^\varepsilon(t)\right)=\left(\tilde{z}^\varepsilon\left(\frac{t}{\varepsilon}\right),\tilde{\zeta}^\varepsilon\left(\frac{t}{\varepsilon}\right)\right),
\end{equation}
which is equal in law to $(z^\vep,\zeta^\vep)$.
Note that the new processes satisfy the same equations, with the same regularities thanks to Proposition \ref{GronwallH} and \eqref{boundglobalwp}. We also define 
\begin{equation}\label{Ktilde}
\tilde{K}_t^\varepsilon=\dbn{\dr_x\tilde{u}^\varepsilon}_{L^2}^2 +\frac{1}{2}\dbn{\tilde{m}^\varepsilon}_{L^2}^2+\frac{1}{2}\|\varepsilon\tilde{V}^\varepsilon\|_{L^2}^2.
\end{equation}
Now, in order to identify  
the limit of the processes $u^\varepsilon$ and $m^\varepsilon$, 
we first prove a bound on the new random variables which is similar to the bounds obtained in Lemmas \ref{borneespK}. 
Note that the set $E$ plays here an important role.

\begin{prop}\label{borneespKtilde}
For $\tilde{K}^\varepsilon$ and $\tilde{V}^\varepsilon$ defined in equations \eqref{Ktilde} and \eqref{defderiveesSko},   there exists a stopping time  $\tilde{\tau}^\varepsilon$ and a constant $C(T)>0$ independent of $\varepsilon$ such that:
\begin{equation}
\mathbb{E}\left[\sup_{t\in[0,\tilde{\tau}^\varepsilon\wedge T]}(\tilde{K}^\varepsilon_t)^2+\alpha\|\sqrt{\varepsilon}\tilde{V}^\varepsilon\|_{L^2(0,\tilde{\tau}^\varepsilon\wedge T;L^2)}^2\right]\leqslant C(T).
\end{equation}
Besides, $\tilde  \tau^\vep$  converges in probability to $T$ as $\varepsilon$ goes to $0$.
\end{prop}
\begin{proof}
The only difference between the system satisfied by $(u^\varepsilon,m^\varepsilon)$ and the one satisfied by $(\tilde{u}^\varepsilon,\tilde{m}^\varepsilon)$ is that the process $z$ is replaced by a process $\tilde{z}^\varepsilon$ which a priori depends on $\varepsilon$.  If it were not the case, we could do the same computations as in section \ref{sectionenergie} and easily recover the above bound. 
Let us introduce
\begin{equation}\label{deftaudeltatilde}
\tilde{\tau}_\delta^\varepsilon = \inf \left\{t: \dbn{\tilde{\z}^\varepsilon}_{H^3\cap\dot{H}^{-3}} + \|\tilde{\zz}^\varepsilon\|_{H^3\cap\dot{H}^{-3}} \geqslant \varepsilon^{-\delta}\right\}.
\end{equation}
Note that the Skorohod Theorem ensures that $(\tilde{z}^\varepsilon, \tilde \zeta^\vep)$ converges a.s. to $(\tilde{z},\tilde \zeta)$ in $E$, so that  for any $\delta'<\delta$
 and for a.e. $\omega\in\tilde{\Omega}$, there exists a constant $C(\omega)$ such that for $\varepsilon$ small enough, 
$$\dbn{\tilde{z}^\varepsilon(t)}_{H^3\cap\dot{H}^{-3}} + \|\tilde{\zeta}^\varepsilon(t)\|_{H^3\cap\dot{H}^{-3}}\leqslant \dbn{\tilde{z}(t)}_{H^3\cap\dot{H}^{-3}}  + \|\tilde{\zeta}(t)\|_{H^3\cap\dot{H}^{-3}}+C(\omega)(1+\vert t\vert^{\delta'}).$$
Besides, $\tilde{z}$ and $z$ have the same law, so we deduce from the above bound and \eqref{borne z zeta} that
there exist two random variables $\tilde{Z_1},\tilde{Z_2}$ such that for $\vep>0$ sufficiently small,
 $$\dbn{\tilde{z}^\varepsilon(t)}_{H^3\cap\dot{H}^{-3}} + \|\tilde{\zeta}^\varepsilon(t)\|_{H^3\cap\dot{H}^{-3}}\leqslant \tilde{Z_1}+\vert t\vert^{\delta'} \tilde{Z_2}.$$ 
 Thus the stopping time $\tilde{\tau}_\delta^\varepsilon$ converges a.s. to infinity when $\varepsilon$ goes to $0$.  
\end{proof}

Note that the bound in Proposition \ref{borneespKtilde} implies the weak convergence, up to the extraction of a subsequence, of the stopped 
processes: there is a subsequence, still denoted $(\tilde{u}^\varepsilon_{\tilde{\tau}^\varepsilon}, \tilde{m}^\varepsilon_{\tilde{\tau}^\varepsilon}, \sqrt{\varepsilon}\tilde{V}^\varepsilon_{\tilde{\tau}^\varepsilon})$, which converges in $L^4(\tilde{\Omega},L^\infty(0,T;H^1))\times L^4(\tilde{\Omega},L^\infty(0,T;L^2))\times L^2(\tilde{\Omega}\times [0,T] \times \ER)$ to $(\tilde{u},\tilde{m},\tilde{V})$ weak star.  
The next two lemmas, which are adapted from \cite{AdAd88}, give precisions on those weak limits.
\begin{lemma} \label{weaklimmtilde}
The weak limit $\tilde{m}$ of the stopped process $\tilde{m}^\varepsilon_{\tilde{\tau}^\varepsilon}$ is equal to $-\vert \tilde{u}\vert^2$.
\end{lemma}
\begin{proof}
First, it is clear that $\vert\tilde{u}^\varepsilon_{\tilde{\tau}^\varepsilon}\vert^2$ converges a.s. to $\vert\tilde{u}\vert^2$ in $L^\infty(0,T,H^s_{loc})$.  Now, let $\varphi$ be a test function in $L^2(\tilde{\Omega};\mathcal{D}((0,T)\times \ER))$. Integrating in time equation \eqref{eqV} gives: 
\begin{align*}
&\mathbb{E}\Big[ \langle\varepsilon^2(\tilde{V}^\varepsilon_{\tilde{\tau}^\varepsilon}-\tilde V^\vep(0)),\mathbf{1}_{[0,\tilde \tau^\vep)}\varphi\rangle 
+\alpha \langle \varepsilon \int_0^t \tilde{V}^\varepsilon_{\tilde{\tau}^\varepsilon}(s) ds, \mathbf{1}_{[0,\tilde \tau^\vep)}\varphi\rangle\Big]\\
& =-\mathbb{E} \langle\ \int_0^t \dr_x\left(\tilde{m}^\varepsilon_{\tilde{\tau}^\varepsilon}(s) +\vert\tilde{u}^\varepsilon_{\tilde{\tau}^\varepsilon}\vert^2(s)\right) ds, \mathbf{1}_{[0,\tilde \tau^\vep)} \varphi\rangle.
\end{align*}
Now, $\sqrt{\vep} \tilde V^\vep_{\tilde \tau^\vep}$ , $\sqrt{\vep} \int_0^t \tilde V^\vep_{\tilde \tau^\vep}(s)ds$  and 
$\mathbf{1}_{[0,\tilde \tau^\vep)} \varphi$ are bounded in $L^2(\tilde{\Omega}\times [0,T] \times \ER)$, and we deduce that
the terms on the left hand side of the above equation tend to $0$ as $\vep$ goes to $0$.
On the other hand, 
\begin{align*}
 \mathbb{E} \langle\ \int_0^t \dr_x\tilde{m}^\varepsilon_{\tilde{\tau}^\varepsilon}(s)ds, \mathbf{1}_{[0,\tilde \tau^\vep)} \varphi\rangle
= \mathbb{E}\langle \dr_x M^\vep, \mathbf{1}_{[0,\tilde \tau^\vep)} \varphi\rangle
\end{align*}
converges to $\mathbb{E}\langle \dr_x \tilde M, \varphi\rangle$, by dominated convergence, thanks to the boundedness of 
$\tilde{m}^\varepsilon_{\tilde{\tau}^\varepsilon}$ in $L^{\infty}(\tilde \Omega; L^\infty(0,T;L^2_x))$ and the fact that 
$\dr_x \tilde M^\vep$ converges a.s. to $\dr_x \tilde M$ in $L^{\infty}(0,T;H^{-\sigma-1}_{loc})$, while $\mathbf{1}_{[0,\tilde \tau^\vep)} \varphi$
converges in probability to $\varphi$ in $L^1(0,T;H^{\sigma+1})$. In the same way, \linebreak $\mathbb{E} \langle \int_0^t \dr_x (|\tilde u^\vep_{\tilde \tau^\vep}|^2,
\mathbf{1}_{[0,\tilde \tau^\vep)} \varphi \rangle$ converges to $\mathbb{E} \langle \int_0^t \dr_x (|\tilde u|^2) ds, \varphi \rangle$,
and we deduce 
$$\dr_x(\tilde{m}+\vert\tilde{u}\vert^2)=0 \text{ a.s.  in } \mathcal{D}'((0,T)\times \ER ).$$
Since $\tilde{m},\vert\tilde{u}\vert^2\in L^2(\tilde{\Omega},L^\infty(0,T;L^2))$, it follows that $\tilde{m}+\vert\tilde{u}\vert^2=0$.
\end{proof}

\begin{lemma}\label{convweakmu}
The process $\tilde{m}^\varepsilon_{\tilde{\tau}^\varepsilon}\tilde{u}^\varepsilon_{\tilde{\tau}^\varepsilon}$ converges weakly to $-\vert\tilde{u}\vert^2\tilde{u}$ in $L^{2}(\Omega,L^2(0,T;H^{-1}))$. 
\end{lemma}
\begin{proof}
 Let  $\varphi\in L^{2}(\Omega,L^2(0,T;H^1)) $ such that $\text{supp } \varphi \subset [0,T]\times [-R,R]$, a.s., for some $R>0$. Then,
\begin{align*}
\mathbb{E}\left[\int_0^T\int_\ER\left(\tilde{m}^\varepsilon_{\tilde{\tau}^\varepsilon}\tilde{u}^\varepsilon_{\tilde{\tau}^\varepsilon}+\vert\tilde{u}\vert^2\tilde{u}\right)\varphi dxdt\right] =& \mathbb{E}\left[\int_0^T\int_\ER\tilde{m}^\varepsilon_{\tilde{\tau}^\varepsilon}\left(\tilde{u}^\varepsilon_{\tilde{\tau}^\varepsilon}-\tilde{u}\right)\varphi \, dx dt\right] \\&+ \mathbb{E}\left[\int_0^T\int_\ER\left(\tilde{m}^\varepsilon_{\tilde{\tau}^\varepsilon}+\vert\tilde{u}\vert^2\right)\varphi  \tilde{u} \,dx dt\right].
\end{align*}
Note that $\tilde{u}^\varepsilon_{\tilde{\tau}^\varepsilon}$ converges to $\tilde u$ in $L^4(\tilde{\Omega},L^\infty(0,T;L^2(-R,R)))$, thanks to
the a.s. convergence in $L^\infty(0,T;H^s_{loc})$ and the uniform integrability given by the conservation of $L^2$-norm. Since the first term in
the right hand side above is easily bounded by
\begin{align*}
&\dbn{\tilde{m}^\varepsilon_{\tilde{\tau}^\varepsilon}}_{L^4(\tilde{\Omega},L^\infty(0,T;L^2))} \dbn{\varphi}_{L^2(\tilde{\Omega},L^2(0,T;H^1))}\dbn{\tilde{u}^\varepsilon_{\tilde{\tau}^\varepsilon}-\tilde{u}}_{L^4(\tilde{\Omega},L^\infty(0,T;L^2(-R,R)))}T^\frac{1}{2},
\end{align*}
we deduce that this term tends to $0$ with $\vep$. 
On the other hand, the conservation of $L^2$-norm for $\tilde{u}$ implies  that  $\varphi\tilde{u}\in L^2(\Omega,L^1(0,T;L^2))$ and Lemma
\ref{weaklimmtilde} allows to
conclude that the second term in the right hand side above converges to $0$ with $\vep$. 
\end{proof}
The above lemmas will be useful in the next subsection. 

\subsection{Martingale Problem} \label{martingale problem}

In this subsection,
we use the Perturbed Test Function method to identify the limit generator of $\tilde u$,  and study the limit as $\vep$ goes to $0$ of the martingale problem
associated to $(\tilde{u}^\varepsilon_t,  \tilde{\z}^\varepsilon_t,\tilde{\zz}^\varepsilon_t)$. Since $\tilde u \in C([0,T]; H^s_{loc})$, a.s.,
we only need quadratic test functions to identify the generator, thanks to the martingale representation theorem (see  \cite[Theorem 8.2]{DPZ14}).
			
			We define the function $\varphi:u\mapsto (u,h)^\ell$ for $h\in H^1$ a fixed function with compact support,  and $\ell=1,2$.  
First we compute, as in Section \ref{sectionenergie}:
\begin{align*}
(\mathcal{L}^\varepsilon\varphi)(\tilde{u}^\varepsilon, \tilde m^\vep,\tilde{\z}^\varepsilon) =\ell(\tilde{u}^\varepsilon,h)^{\ell-1} (i\dr_x^2\tilde{u}^\varepsilon-i\tilde{m}^\varepsilon\tilde{u}^\varepsilon,h)-\frac{\ell}{\sqrt{\varepsilon}}(\tilde{u}^\varepsilon,h)^{\ell-1}(i\tilde{\z}^\varepsilon\tilde{u}^\varepsilon,h).
\end{align*}
The first corrector is then given by (see \eqref{firstcorr}):
\begin{equation}\label{varphi1}
\begin{split}
\varphi_1(\tilde{u}^\varepsilon,\tilde{\z}^\varepsilon,\tilde{\zz}^\varepsilon)
&=\ell(\tilde{u}^\varepsilon,h)^{\ell-1}(i\tilde{u}^\varepsilon(\dr_x^2)^{-1}(\tilde{\zz}^\varepsilon+\alpha\tilde{\z}^\varepsilon),h),
\end{split}
\end{equation}
and the second corrector (see \eqref{secondcorr}):
\begin{equation}\label{varphi2}
\begin{split}
\varphi_2(\tilde{u}^\varepsilon,\tilde{\z}^\varepsilon,\tilde{\zz}^\varepsilon)=&-\ell(\tilde{u}^\varepsilon,h)^{\ell-1}\left(\tilde{u}\MM^{-1}\left(\tilde{\z}^\varepsilon\MM^{-1}\tilde{\z}^\varepsilon-\mathbb{E}_\nu^\vep\left[z\MM^{-1}z\right]\right),h\right) \\
+\ell(\ell-1)&\MM^{-1}\left((i\tilde{\z}^\varepsilon\tilde{u}^\varepsilon,h)(i\tilde{u}^\varepsilon\MM^{-1}\tilde{\z}^\varepsilon,h)-\mathbb{E}_\nu\left[(iz\tilde{u}^\varepsilon,h)(i\tilde{u}^\varepsilon\MM^{-1}z,h)\right]\right).
\end{split}
\end{equation}
Finally, defining $\varphi^\varepsilon=\varphi+\sqrt{\varepsilon}\varphi_1+\varepsilon\varphi_2$, we get: 
\begin{equation} \label{genLeps}
\begin{split}
(\mathcal{L}^\varepsilon&\varphi^\varepsilon)(\tilde{u}^\varepsilon,\tilde m^\vep, \tilde{\z}^\varepsilon,\tilde{\zz}^\varepsilon) \\
= & \, \ell(\tilde{u}^\varepsilon,h)^{\ell-1}(i\dr_x^2\tilde{u}^\varepsilon-i\tilde{m}^\varepsilon\tilde{u}^\varepsilon,h) +\ell(\tilde{u}^\varepsilon,h)^{\ell-1}(\tilde{u}^\varepsilon\mathbb{E}_\nu\left[ z\MM^{-1}z\right],h) \\
&-\ell(\ell-1)\mathbb{E}_\nu\left[(iz\tilde{u}^\varepsilon,h)(i\tilde{u}^\varepsilon\MM^{-1}z,h)\right] \\
&+\sqrt{\varepsilon}D_u\varphi_1(i\dr_x^2\tilde{u}^\varepsilon-i\tilde{m}^\varepsilon\tilde{u}^\varepsilon) \\
&+\sqrt{\varepsilon} D_u\varphi_2(-i\tilde{\z}^\varepsilon\tilde{u}^\varepsilon) + \varepsilon D_u\varphi_2(i\dr_x^2\tilde{u}^\varepsilon-i\tilde{m}^\varepsilon\tilde{u}^\varepsilon).
\end{split}
\end{equation}
Taking formally the limit in $\vep$ in the above expression, and using Lemma \ref{weaklimmtilde}, we obtain a candidate for the limit generator
given by
\begin{equation} \label{limitgenerator}
\begin{split}
\mathcal{L}\varphi(\tilde{u})=&\ell(\tilde{u},h)^{\ell-1}(i\dr_x^2\tilde{u}+i\vert \tilde{u}\vert^2\tilde{u},h) + \ell(\tilde{u},h)^{\ell-1}(\tilde{u}\mathbb{E}_\nu\left[z\MM^{-1}z\right],h) \\&- \ell(\ell-1)\mathbb{E}_\nu\left[(i\tilde{u}\MM^{-1}z,h)(i\tilde{u}z,h)\right].
\end{split}
\end{equation}
We recall that here, $\varphi(\cdot)=(\cdot,h)^\ell$, $\ell=1,2$,  with $h\in H^1$. Note that the above expression is well defined for $\tilde{u}\in H^3$. 
The next proposition gives some precision on the above convergence.
\begin{prop}\label{propFTP}
Let $\varphi(u)=(u,h)^\ell$,  $\ell=1,2$,  where $h\in H^1$ with compact support.  Let $\varphi_1,\varphi_2$ be the two correctors defined in equations \eqref{varphi1} and \eqref{varphi2}, and let $\varphi^\varepsilon=\varphi+\sqrt{\varepsilon}\varphi_1+\varepsilon\varphi_2$. Then:
\begin{enumerate}
\item 
There is a constant $C$ depending on $\|\phi\|_{\mathcal{L}(L^2;H^1\cap \dot H^{-4})}$ such that
$$\vert \varphi^\varepsilon(\tilde{u}^\varepsilon_{\tilde{\tau}^\varepsilon}, \tilde{\z}^\varepsilon_{\tilde{\tau}^\varepsilon},\tilde{\zeta}^\varepsilon_{\tilde{\tau}^\varepsilon})-\varphi({\tilde{u}^\varepsilon_{\tilde{\tau}^\varepsilon}})\vert \leqslant C \sqrt{\varepsilon}\dbn{\tilde{u}^\varepsilon_{\tilde{\tau}^\varepsilon}}^\ell_{L^2}\dbn{h}^\ell_{L^2} (1+\dbn{\tilde{\z}^\varepsilon_{\tilde{\tau}^\varepsilon}}_E^2+\|\tilde{\zz}^\varepsilon_{\tilde{\tau}^\varepsilon}\|_E^2 ), $$ 
where $E$ is defined in Equation \eqref{defE}.

\item The quantity
$ \mathcal{L}^\varepsilon\varphi^\varepsilon(\tilde{u}^\varepsilon_{\tilde{\tau}^\varepsilon}, \tilde m^\vep_{\tilde \tau^\vep}, \tilde{\z}^\varepsilon_{\tilde{\tau}^\varepsilon},\tilde{\zz}^\varepsilon_{\tilde{\tau}^\varepsilon})-\mathcal{L}\varphi(\tilde{u})$
converges to $0$ in $L^2(\tilde \Omega; L^{\infty}(0,T))$ weak star, as $\vep$ goes to $0$.

\item
Taking $\ell=1$, the quantity  $M^\varepsilon_{t}$ defined by 
\begin{equation}\label{defMepsilon}
\left(M^\varepsilon_{t},h\right) = \varphi^\varepsilon(\tilde{u}^\varepsilon_t,  \tilde{\z}^\varepsilon_t,\tilde{\zz}^\varepsilon_t)-\varphi^\varepsilon(\tilde{u}_0,\tilde{\z}_0,\tilde{\zz}_0) -\int_0^{t}\mathcal{L}^\varepsilon\varphi^\varepsilon(\tilde{u}^\varepsilon_s, 
\tilde m^\vep_s, \tilde{\z}^\varepsilon_s,\dr_t\tilde{\zz}^\varepsilon_s)\,ds
\end{equation}
 is a local martingale for the filtration $\tilde{\mathcal{F}}^\varepsilon$ generated by  the process {$(\tilde{u}^\varepsilon)$}.
\end{enumerate}
\end{prop}

\begin{proof}
The first point of the Proposition is clear in view of \eqref{varphi1}, \eqref{varphi2} and Proposition~\ref{coroborneLinf}.
For the second point,  we take the difference between \eqref{genLeps} and \eqref{limitgenerator}, and consider each term of the resulting
expression.
We start with the terms of order $\frac{1}{2}$ and 1 in $\varepsilon$; using again \eqref{varphi1}, \eqref{varphi2} and 
Proposition \ref{coroborneLinf}, together with
\eqref{deftaudeltatilde},
$$\sqrt{\varepsilon}D_u\varphi_1(i\dr_x^2\tilde{u}^\varepsilon_{\tilde{\tau}^\varepsilon}-i\tilde{m}^\varepsilon_{\tilde{\tau}^\varepsilon}\tilde{u}^\varepsilon_{\tilde{\tau}^\varepsilon})\leqslant C\varepsilon^{\frac{1}{2}-\delta}\left(\dbn{\tilde{u}^\varepsilon_{\tilde{\tau}^\varepsilon}}_{H^1} + \dbn{\tilde{m}^\varepsilon_{\tilde{\tau}^\varepsilon}}_{L^2}\dbn{\tilde{u}^\varepsilon_{\tilde{\tau}^\varepsilon}}_{L^2}\right)\dbn{\tilde{u}^\varepsilon_{\tilde{\tau}^\varepsilon}}^{\ell-1}_{L^2}\dbn{h}_{H^1}^\ell,$$
and
$$\sqrt{\varepsilon}D_u\varphi_2(-i\tilde{\z}^\varepsilon_{\tilde{\tau}^\varepsilon}\tilde{u}^\varepsilon_{\tilde{\tau}^\varepsilon})
\leqslant C\varepsilon^{1-3\delta} \dbn{\tilde{u}^\varepsilon_{\tilde{\tau}^\varepsilon}}_{L^2}^\ell\dbn{h}^{\ell}_{L^2}.$$
Besides,
\begin{align*}
\varepsilon D_u\varphi_2(i\dr_x^2\tilde{u}^\varepsilon_{\tilde{\tau}^\varepsilon}-i\tilde{m}^\varepsilon_{\tilde{\tau}^\varepsilon}\tilde{u}^\varepsilon_{\tilde{\tau}^\varepsilon}) 
&\leqslant C \varepsilon^{1-2\delta} \left(\dbn{\tilde{u}^\varepsilon_{\tilde{\tau}^\varepsilon}}_{H^1}+\dbn{\tilde{m}^\varepsilon_{\tilde{\tau}^\varepsilon}}_{L^2}\dbn{\tilde{u}^\varepsilon_{\tilde{\tau}^\varepsilon}}_{L^2}\right)\dbn{\tilde{u}^\varepsilon_{\tilde{\tau}^\varepsilon}}_{L^2}^{\ell-1}\dbn{h}_{H^1}^\ell,
\end{align*}
and we may conclude thanks to Proposition \ref{borneespKtilde} that these three terms tend to 0 when $\varepsilon$ goes to 0, strongly in
$L^2(\tilde \Omega; L^{\infty}(0,T))$.
Furthermore, the zero order terms may be treated thanks to the strong convergence of $\tilde{u}^\varepsilon_{\tilde{\tau}^\varepsilon}$
to $\tilde u$ in $L^4(\tilde{\Omega},L^\infty(0,T;L^2_{loc}))$ (see the proof of Lemma \ref{convweakmu}) and the weak convergence 
of Lemma \ref{convweakmu}. 
Indeed, the above convergences, together with the weak convergence of $\tilde{u}^\varepsilon_{\tilde{\tau}^\varepsilon}$ to $\tilde u$ in
$L^4(\tilde{\Omega},L^\infty(0,T;H^1))$ allow to get the weak convergence in $L^2(\Omega; L^\infty(0,T))$ of the term
$(\tilde{u}^\varepsilon_{\tilde{\tau}^\varepsilon},h)(i\dr_x^2 \tilde{u}^\varepsilon_{\tilde{\tau}^\varepsilon} -i\tilde{m}^\varepsilon_{\tilde{\tau}^\varepsilon}\tilde{u}^\varepsilon_{\tilde{\tau}^\varepsilon},h)$  to $(\tilde{u},h)(i\dr_x^2 \tilde u + i\vert\tilde{u}\vert^2\tilde{u},h)$,
while the two remaining zero order terms converge strongly in $L^2(\tilde \Omega;L^\infty(0,T))$.

The last point of the Proposition stems from a simple application of the It\^o formula, given the smoothness of the processes 
$\tilde{u}^\varepsilon,\tilde{m}^\varepsilon,\tilde{z}^\varepsilon,\tilde{\zeta}^\varepsilon$.
\end{proof}

\subsection{Convergence in law}

We are now in position to prove the convergence in law of the process $\tilde u^\vep$ to the unique solution of \eqref{Schrostoc} with
$u(0)=u_0$.

\vskip 0.1 in 
\noindent
{\em Identification of the limit generator.}
We define 
\begin{equation} \label{noyau}
k(x,y)=\mathbb{E}_\nu\left[z(x)\MM^{-1}z(y)+z(y)\MM^{-1}z(x)\right],
\end{equation}
 and $F(x)=k(x,x)$. The next lemma relates the kernel $k$ and the Hilbert-Schmidt operator $\phi$.
\begin{lemma}\label{k(x,y)}
For $x,y\in \ER$ and $(z,\zeta)$ solution of equation \eqref{equationz2}, we have:
\begin{equation}\label{expressionk(x,y)}
-k(x,y)=\sum_{k\in\EN}(\dr_x^2)^{-1}(\phi e_k)(x)(\dr_x^2)^{-1}(\phi e_k)(y),
\end{equation}
where we recall that $(e_k)_{k\in \EN}$ is a complete orthonormal system in $L^2(\ER;\ER)$.
\end{lemma} 
The proof of Lemma \ref{k(x,y)} is technical and is postponed to the Appendix. Now, we recall that in order to identify the limit generator 
$\mathcal{L}$, we only need to apply it to quadratic test functions. Considering first  $\varphi(u)=(u,h)$, with
$h\in H^1$ with compact support, we use \eqref{limitgenerator} and \eqref{expressionk(x,y)} to get:
\begin{align} \label{Llimlin}
\mathcal{L}\varphi(\tilde{u}) &= \big(i\dr_x^2\tilde{u}+i\vert\tilde{u}\vert^2\tilde{u}-\frac{1}{2}\tilde{u}F(x),h\big) 
=(D_{u}\varphi)\big(i\dr_x^2\tilde{u}+i\vert\tilde{u}\vert^2\tilde{u}-\frac{1}{2}\tilde{u}F(x)\big),
\end{align} 
where we recall that $F(x)=\sum_{k\in\EN}\big((\dr_x^2)^{-1}\phi e_k\big)^2(x)$.
Next, considering $\varphi(u)=(u,h)^2$ gives:
\begin{align} 
\mathcal{L}\varphi(\tilde{u})=
&\, (D_{u}\varphi )\big(i\dr_x^2\tilde{u}+i\vert\tilde{u}\vert^2\tilde{u}-\frac{1}{2}\tilde{u}F(x)\big) \nonumber\\
&-\int_\ER\int_\ER   Re    \left(i\tilde{u}(x)\bar{h}(x)\right)   Re    \left(i\tilde{u}(y)\bar{h}(y)\right)k(x,y)dxdy. \label{Llimquad}
\end{align}
Note that the second term of the right-hand-side above may be written, thanks to Lemma~\ref{k(x,y)} as:
\begin{align*}
\sum_{k\in\EN}\left(i\tilde u(\dr_x^2)^{-1}(\phi e_k),h\right)^2 
=&\frac{1}{2}\text{Tr}\, (i\tilde u(\dr_x^2)^{-1}\phi)^*(D^2_u\varphi)(i\tilde u(\dr_x^2)^{-1}\phi).
\end{align*}
We deduce that
\begin{equation}\label{Llim}
\begin{split}
\mathcal{L}\varphi(\tilde{u})=&(D_{u}\varphi) \big(i\dr_x^2\tilde{u}+i\vert\tilde{u}\vert^2\tilde{u}-\frac{1}{2}\tilde{u}F(x)\big)\\
&+\frac{1}{2}\text{Tr}\, (i\tilde u(\dr_x^2)^{-1}\phi)^*(D^2_u\varphi)(i\tilde u(\dr_x^2)^{-1}\phi),
\end{split}
\end{equation}
which is the generator of the transition semi-group associated with equation \eqref{Schrostoc}.

\vskip 0.1 in
\noindent
{\em Convergence.}
We first prove the following lemma.
\begin{lemma}\label{martlimite}
Let $\varphi(u)=(u,h)$ where $h\in H^1$ with compact support. Then the process $M_t$ defined by:
\begin{equation}\label{defMt}
(M_t,h)=\varphi(\tilde{u}_t)-\varphi(\tilde{u}_0)-\int_0^t\mathcal{L}\varphi(\tilde{u}_s)ds
\end{equation}
is a martingale for the filtration $\mathcal{G}_s$ generated by $\tilde{u}_s$. 
\end{lemma}
\begin{proof}
Let $g\in \mathcal{C}_b\left((H^s_{loc})^n\right)$, let {$0\leqslant s_1<s_2<\cdots<s_n=\sigma <t$}, and let $\varphi_1$ and $\varphi_2$ be the correctors
defined in \eqref{varphi1} and \eqref{varphi2} respectively, with $\ell=1$.  Then, thanks to Proposition~\ref{propFTP}, 
$$\Es \big[(M^\varepsilon_{\tilde{\tau}^\varepsilon}(t)-M^\varepsilon_{\tilde{\tau}^\varepsilon}(\sigma),h)
g\left(\tilde{u}^\varepsilon_{\tilde{\tau}^\varepsilon}(s_1),\cdots, \tilde{u}^\varepsilon_{\tilde{\tau}^\varepsilon}(s_n)\right) \big]=0,$$
 so we deduce that
\begin{align*}
&\mathbb{E}\Big[\big(\varphi(\tilde{u}_t)-\varphi(\tilde{u}_\sigma)-\int_\sigma^t\mathcal{L}\varphi(\tilde{u}_r)dr\big)g\left(\tilde{u}^\varepsilon_{\tilde{\tau}^\varepsilon}(s_1),\cdots,\tilde{u}^\varepsilon_{\tilde{\tau}^\varepsilon}(s_n)\right)\Big] \\
&=\mathbb{E}\Big[\big(\left(\varphi(\tilde{u}_t)-\varphi(\tilde{u}^\varepsilon_{\tilde{\tau}^\varepsilon}(t))\right) -\left(\varphi(\tilde{u}_\sigma)-\varphi(\tilde{u}^\varepsilon_{\tilde{\tau}^\varepsilon}(\sigma))\right)\\
 &-\sqrt{\varepsilon}\left(\varphi_1(\tilde{u}^\varepsilon_{\tilde{\tau}^\varepsilon}(t))-\varphi_1(\tilde{u}^\varepsilon_{\tilde{\tau}^\varepsilon}(\sigma))\right) -\varepsilon\left(\varphi_2(\tilde{u}^\varepsilon_{\tilde{\tau}^\varepsilon}(t))-\varphi_2(\tilde{u}^\varepsilon_{\tilde{\tau}^\varepsilon}(\sigma))\right) \\
 &-\big(\int_\sigma^{\sigma\wedge\tilde{\tau}^\varepsilon}\mathcal{L}\varphi(\tilde{u}_r)dr-\int_t^{t\wedge\tilde{\tau}^\varepsilon}\mathcal{L}\varphi(\tilde{u}_r)dr\big) \\
 &{-\int_{\sigma\wedge\tilde{\tau}^\varepsilon}^{t\wedge\tilde{\tau}^\varepsilon}\mathcal{L}\varphi(\tilde{u}_r)-\mathcal{L}^\varepsilon\varphi^\varepsilon(\tilde{u}^\varepsilon_{\tilde{\tau}^\varepsilon}(r),\tilde{\z}^\varepsilon_{\tilde{\tau}^\varepsilon}(r),\tilde{\zz}^\varepsilon_{\tilde{\tau}^\varepsilon}(r))dr\big)g(\tilde{u}^\varepsilon_{\tilde{\tau}^\varepsilon}(s_1),\cdots,\tilde{u}^\varepsilon_{\tilde{\tau}^\varepsilon}(s_n))\Big] }\\
 &=\mathbb{E}\left[T_1+T_2+T_3+T_4\right].
\end{align*}
Note that, for convenience,  we did not write the dependence of $\varphi_1$ and $\varphi_2$ on $\tilde{m}^\varepsilon, \tilde{\z}^\varepsilon$ 
and $\tilde{\zz}^\varepsilon$.
The convergence of $\tilde{u}^\vep_{\tau^\vep}$ to $\tilde u$, the boundedness of $g$, and Proposition \ref{propFTP} imply
$$\lim_{\varepsilon\to0}\mathbb{E}\left[T_1+T_2\right]=0.$$
 Regarding $T_3$, in view of 
\eqref{Llimlin} and the boundedness of $g$,
\begin{align*}
\mathbb{E}\Big[\big(\int_\sigma^{\sigma\wedge\tilde{\tau}^\varepsilon} &\mathcal{L}\varphi(\tilde{u}_r)dr\big) 
 g(\tilde{u}^\varepsilon_{\tilde{\tau}^\varepsilon}(s_1),\cdots,\tilde{u}^\varepsilon_{\tilde{\tau}^\varepsilon}(s_n))\Big] 
\\
\leqslant & \, C \dbn{h}_{H^1} \mathbb{E}\big[(\sigma\wedge\tilde{\tau}^\varepsilon-\sigma)\dbn{\tilde{u}}_{L^\infty([0,T],H^1)}\big] 
\leqslant C \, \mathbb{E}\big[(\sigma\wedge\tilde{\tau}^\varepsilon-\sigma)^2\big]^\frac{1}{2},
\end{align*} 
since $\tilde{u}\in L^4(\tilde{\Omega},L^\infty([0,T],H^1))$.  
The  same bound obviously holds for the other term of $T_3$, and since $\tilde{\tau}^\varepsilon$ converges in probability to $T$, we deduce that
 $\lim_{\vep \to 0} \mathbb{E}[T_3]=0$. Finally, the convergence of $\Es[T_4]$ to zero is obtained thanks to Proposition \ref{propFTP} (2),
 the strong convergence  of $g(\tilde{u}^\varepsilon_{\tilde{\tau}^\varepsilon}(s_1),\cdots,\tilde{u}^\varepsilon_{\tilde{\tau}^\varepsilon}(s_n))$ to $g(\tilde u(s_1),\cdots,\tilde u(s_n))$ in $L^4(\tilde \Omega)$ and of $\mathbf{1}_{[\sigma \wedge \tilde \tau^\vep,t\wedge  \tilde \tau^\vep]}$
 to $\mathbf{1}_{[\sigma, t]}$ in $L^4(\tilde \Omega;L^1(0,T))$.
 
We thus proved that 
\begin{equation}
\lim_{\varepsilon\to0}\mathbb{E}\Big[\big(\varphi(\tilde{u}_t)-\varphi(\tilde{u}_\sigma)-\int_\sigma^t\mathcal{L}\varphi(\tilde{u}_r)dr\big)g\left(\tilde{u}^\varepsilon_{\tilde{\tau}^\varepsilon}(s_1),\cdots,\tilde{u}^\varepsilon_{\tilde{\tau}^\varepsilon}(s_n)\right)\Big]  =0,
\end{equation}
and the continuity of $g$ and the a.s. convergence of $\tilde{u}^\varepsilon_{\tilde{\tau}^\varepsilon}$ allow to conclude
that the process $M_t$ is a martingale for the filtration $\mathcal{G}_s$ generated by $\tilde{u}_s$.
\end{proof}

Similarly to Lemma \ref{martlimite}, we can take the limit  in $\varepsilon$ of the quadratic variation of $M^\varepsilon$ defined  in \eqref{defMepsilon} to get
\begin{equation} \label{varquad}
\left(\langle M,M\rangle_th,h\right)=\int_0^t \left(\mathcal{L}(\varphi^2)-2\varphi\mathcal{L}\varphi\right)(\tilde{u}_s)ds,
\end{equation}
where, still, $\varphi(u)=(u,h)$.
Note that this requires the use of Proposition \ref{propFTP}, as in the proof of Lemma \ref{martlimite}, but with $\ell=1$ and $\ell=2$,
together with an expression similar to \eqref{varquad} for $M^\vep_t$ and $\mathcal{L}^\vep$.
Equation \eqref{Llim}  then implies that 
$$\left(\mathcal{L}(\varphi^2)-2\varphi\mathcal{L}\varphi\right)(\tilde{u}) =\frac{1}{2}\text{Tr}\, (i\tilde u(\dr_x^2)^{-1}\phi)^*(D^2_u(\varphi^2))(i\tilde u(\dr_x^2)^{-1}\phi).$$
Thus,  the martingale $$M_t=\tilde{u}_t-u_0 -\int_0^t \big(i\dr_x^2\tilde{u}_s +i\vert\tilde{u}_s\vert^2\tilde{u}_s-\frac{1}{2}\tilde{u}_s F(x) \big)ds$$
has the quadratic variation:
$$\int_0^t \text{Tr}\left(i\tilde{u}_s(\dr_x^2)^{-1}\phi\right)\left(i\tilde{u}_s(\dr_x^2)^{-1}\phi\right)^*ds.$$
Thus,  using the martingale representation theorem, and up to enlarge the probability space, there exists a cylindrical Wiener process $\bar{W}$ such that 
 $$M_t=\int_0^ti\tilde{u}_s(\dr_x^2)^{-1}\phi d\bar{W}_s,$$
 so $\tilde{u}$ is a martingale solution of \eqref{Schrostoc}.
 
 \vskip 0.1 in
 \noindent
{\em Uniqueness.}
So far we have proved the convergence in law to Equation \eqref{Schrostoc} for a subsequence,  we need to conclude with the convergence of the whole sequence.
This is however a simple consequence of the fact that $\tilde u$ has paths a.s. in $L^\infty(0,T,H^1)$, and the (obvious) uniqueness
of the solution of \eqref{Schrostoc} with paths in $L^\infty(0,T,H^1)$ a.s.  This, together with the convergence in probability of $\tau^\vep$ to $T$
leads to the convergence in law of $(u^\vep)$ to the solution of \eqref{Schrostoc}.

\subsection{Convergence in probability}

The aim here is to prove that we can actually get a better convergence of the solution $u^\varepsilon$ of \eqref{systemetranslate} to the solution $u$ of \eqref{Schrostoc}, as announced in Theorem \ref{thmfinal}. It is not usually the case in approximation diffusion problems, and it is due here to the particular form of the process $(z^\varepsilon,\zeta^\varepsilon)$.

Note that the correctors $\varphi_1$ and $\varphi_2$ introduced in \eqref{varphi1} and \eqref{varphi2} give us an explicit expression 
of the martingale $M_t^\varepsilon$ defined in \eqref{defMepsilon}. Indeed, the process $(z^\varepsilon,\zeta^\varepsilon)$
being a solution of \eqref{equationzeps},  where $W_t$ is the cylindrical Wiener process introduced in the system \eqref{e0.1},
when we apply the Perturbed Test Function method to $\varphi(u)=(u,h)$ for $h\in H^1$ with compact support, the first corrector 
$\sqrt{\varepsilon}\varphi_1(u^\varepsilon)=\sqrt{\varepsilon}(iu^\varepsilon((\dr_x^2)^{-1}\zeta^\vep+\alpha (\dr_x^2)^{-1} z^\vep),h)$ contributes to the martingale part according to Equation \eqref{equationzeps} satisfied by $(z^\varepsilon,\zeta^\vep)$, adding 
\begin{equation}\label{Xepsilon}
\begin{split}
X^\varepsilon_t=\int_0^t\left(iu^\varepsilon(\dr_x^2)^{-1}\phi dW_s,h\right).
\end{split}
\end{equation}
In the same way,  according to Lemma \ref{inverse2}, the martingale part coming from the second corrector  $\varepsilon \varphi_2(u^\varepsilon)=-\varepsilon(u^\varepsilon\MM^{-1}(z^\varepsilon\MM^{-1}z^\varepsilon-\mathbb{E}_\nu[z\MM^{-1}z],h)$ is  
\begin{equation}\label{Yepsilon}
\begin{split}
Y^\varepsilon_t=& {\sqrt{\varepsilon}}    Re   \!\! \int_{s=0}^t \!\int_\ER u^\varepsilon(s)\bar{h} \! \!\int_{t'=0}^\infty  \!\!\big[ \!\left(S_\alpha(t')(z^\varepsilon, \phi dW_s)\right)_1 \!\!
 (\dr_x^2)^{-1}\!\left(\left(S_\alpha(t')(z^\varepsilon, \zeta^\varepsilon)\right)_2\!+\alpha \!\left(S_\alpha(t')(z^\varepsilon, \zeta^\varepsilon)\right)_1\right) \\
&+\left(S_\alpha(t')(z^\varepsilon, \zeta^\varepsilon)\right)_1 (\dr_x^2)^{-1}\left(\left(S_\alpha(t')(z^\varepsilon, \phi dW_s)\right)_2+\alpha \left(S_\alpha(t')(z^\varepsilon, \phi dW_s)\right)_1\right)\big] dt'dx.
\end{split}
\end{equation}
Thus the previous section and more precisely the third item of Proposition \ref{propFTP} ensures that 
\begin{equation}\label{egalitemartingale}
\varphi^\varepsilon(u^\varepsilon_t,z^\vep_t,\zeta^\vep_t)-\varphi^\varepsilon(u_0,z^\vep_0,\zeta^\vep_0)-\int_0^t\mathcal{L}^\varepsilon\varphi^\varepsilon(u^\varepsilon_s,z^\vep_s,\zeta^\vep_s)ds = X^\varepsilon_t+Y^\varepsilon_t,
\end{equation}
with $\varphi^\varepsilon=\varphi+\sqrt{\varepsilon}\varphi_1+\varepsilon\varphi_2$.
The quadratic variation of $Y^\varepsilon_t$ is estimated thanks to Lemma \ref{inegalites generales} and we obtain:
 $$d\langle Y^\varepsilon,Y^\varepsilon\rangle_{t\wedge \tau^\varepsilon} \leqslant \varepsilon  (1+\varepsilon^{-4\delta}) C_\phi \dbn{u^\varepsilon}^2_{L^2}\dbn{h}_{H^1}^2$$ where $\tau^\varepsilon$ is the stopping time defined in \eqref{deftaudelta}. We deduce from Doob's inequality that
\begin{equation}\label{Yepsi to 0}
\lim_{\varepsilon\to0}\mathbb{E}\Big[\big|\sup_{t\in[0,\tau^\varepsilon]}Y^\varepsilon_t\big|\Big]=0.
\end{equation}

We are now ready to end the proof of Theorem \ref{thmfinal}.

\begin{proof}[Proof of Theorem \ref{thmfinal}]
Let $(u^{\varepsilon_k},m^{\vep_k})_k$ and $(u^{\eta_k},m^{\eta_k})_k$ be two subsequences of the family of solutions
$(u^\varepsilon, m^\vep)_\varepsilon$, of \eqref{systemetranslate}. Then the proof of  Proposition \ref{tensionuplet} ensures that the sequence 
$$(u^{\varepsilon_k}, u^{\eta_k}, M^{\varepsilon_k}, M^{\eta_k}, \sqrt{\varepsilon_k}\mathcal{V}^{\varepsilon_k},\sqrt{\eta_k}\mathcal{V}^{\eta_k},z,\zeta,(\beta_n)_{n\in\EN})_k$$ 
is tight in the space $(\mathcal{C}^0([0,T],H^s_{loc}))^2\times  (\mathcal{C}^0([0,T],H^{-\sigma}_{loc}))^2\times (H^s([0,T],H^{-\sigma}_{loc}))^2\times E\times E\times \mathcal{C}([0,T])^\EN$
 for $s<1,\sigma>0$.  According to Skorohod Theorem, there exists a probability space $(\tilde{\Omega},\tilde{\mathcal{F}},\tilde{\PP})$ 
 and a sequence of random variables on $\tilde{\Omega}$, with the same  laws,  which converges almost surely   in this space as 
 $\varepsilon$ goes to $0$.  We denote respectively by $\tilde{u}$ and $\tilde{v}$ the almost sure limits of the new subsequences 
 $\tilde{u}^{\varepsilon_k}$ and $\tilde{u}^{\eta_k}$.  For any $n\in\EN$ we also denote by $\tilde{\beta}_n$ the limit of the subsequence 
 $(\tilde{\beta}^k_n)$, and we define:$$\tilde{W}^k_t(x)=\sum_{n\in\EN}e_n(x)\tilde{\beta}^k_n(t), \qquad \tilde{W}_t(x)=\sum_{n\in\EN}e_n(x)\tilde{\beta}_n(t).$$
According to section \ref{martingale problem}, and \eqref{Xepsilon}, \eqref{Yepsilon} and \eqref{egalitemartingale}, 
the stopped process $\tilde{u}^{\varepsilon_k}_{\tilde{\tau}^{\varepsilon_k}}$, where $\tilde{\tau}^\varepsilon$ is the stopping 
time which appears in Proposition \ref{borneespKtilde},  satisfies for $\varphi(u)=(u,h)$, where $h\in H^1(\ER)$ with compact support,
\begin{equation}\label{probmart}
\begin{split}
\varphi^{\varepsilon_k}(\tilde{u}^{\varepsilon_k}_{\tilde{\tau}^{\varepsilon_k}}(t),\tilde{z}^{\varepsilon_k}_{\tilde{\tau}^{\varepsilon_k}}(t),
\tilde{\zeta}^{\varepsilon_k}_{\tilde{\tau}^{\varepsilon_k}}(t))- &\varphi^{\varepsilon_k}(\tilde{u}_0,\tilde z^\vep_0,\tilde \zeta^\vep_0)-\int_0^{t\wedge\tilde{\tau}^{\varepsilon_k}}\mathcal{L}^{\varepsilon_k}\varphi^{\varepsilon_k}(\tilde{u}^{\varepsilon_k}_{\tilde{\tau}^{\varepsilon_k}}, \tilde{z}^{\varepsilon_k}_{\tilde{\tau}^{\varepsilon_k}},\tilde{\zeta}^{\varepsilon_k}_{\tilde{\tau}^{\varepsilon_k}})(s)ds \\
=& \int_0^{t\wedge\tilde{\tau}^{\varepsilon_k}}\left(i\tilde{u}^{\varepsilon_k}_{\tilde{\tau}^{\varepsilon_k}}(s)(\dr_x^2)^{-1}\phi d\tilde{W}^{k}_s,h\right) +\tilde{Y}^{\varepsilon_k}_{\tilde{\tau}^{\varepsilon_k}},
\end{split}
\end{equation}
where $\tilde{Y}^{\varepsilon}$ is defined as $Y^\varepsilon$ in \eqref{Yepsilon} but using the new random variables.  The stopped process  $\tilde{u}^{\eta_k}_{\tilde{\tau}^{\eta_k}}$ satisfies the same equation with $\varepsilon_k$ replaced by $\eta_k$.
Proposition \ref{propFTP} ensures that the left hand side of \eqref{probmart}
 converges weakly in $L^2(\tilde \Omega)$ to 
$$(\tilde{u}_t,h)-(\tilde{u_0},h)-\int_0^t\mathcal{L}\varphi(\tilde{u_s})ds,$$ 
where $\mathcal{L}$ is the limit generator defined in \eqref{limitgenerator}.
 We also know by \eqref{Yepsi to 0} that $Y^{\varepsilon_k}_{\tilde{\tau}^{\varepsilon_k}}$ converges in probability to 0. 
 On the other hand, applying \cite[Lemma 2.1]{DbGHTe11} to the sequences $(\tilde{W}^{k})_k$ and 
 $(\tilde{u}^{\varepsilon_k}_{\tilde{\tau}^{\varepsilon_k}}(\dr_x^2)^{-1}\phi)_k$,  the  integral on the right hand side of
 \eqref{probmart} converges in probability to 
 $$\int_0^t\left(i\tilde{u}(s)(\dr_x^2)^{-1}\phi d\tilde{W}_s,h\right).$$ 
 The procedure is the same for $\tilde{u}^{\eta_k}_{\tilde{\tau}^{\eta_k}}$, and finally the limits $\tilde{u}$ and $\tilde{v}$ are both solutions of $$(u_t,h)-(u_0,h)-\int_0^t\left(i\dr_x^2u_s+i\vert u_s\vert^2u_s+\frac{1}{2}u_sF, h\right)ds=\int_0^t\left(iu_s(\dr_x^2)^{-1}\phi d\tilde{W}_s,h\right),$$ with $F(x)=\sum_{k\in\EN}\left((\dr_x^2)^{-1}\phi e_k(x)\right)^2$.  We conclude thanks to uniqueness of the solution of the stochastic Schr\"odinger equation that $\tilde{u}=\tilde{v}$.

Thus,
we have proved that for all subsequences $(u^{\varepsilon_k})_k,(u^{\eta_k})_k$, the couple $(u^{\varepsilon_k},u^{\eta_k})_{k}$ 
converges in law (up to a subsequence) when $\varepsilon$ goes to 0 to a limit $(u,u)$.  
Thanks to the  Gyongy-Krylov argument (see Lemma 1.1 in \cite{GyoKry96}), the sequence $(u^\varepsilon)_\varepsilon$ 
converges in probability as $\varepsilon$ goes to 0 to a limit  $u$, which, again by the above arguments,
is a weak solution of 
\begin{equation}\label{eq4.8}
idu = (-\dr_x^2u-\vert u\vert^2u-\frac{i}{2}uF)dt-u(\dr_x^2)^{-1}\phi dW_t,
\end{equation}
where $W_t$ is the cylindrical Wiener process of \eqref{e0.1}. 
Actually, it is not difficult to prove that $u$ is a mild solution of Equation \eqref{eq4.8}. 

We would like to insist on the fact that here we used the convergence in probability to pass to the limit directly in the martingale term in \eqref{probmart},  because $\lim \tilde{W}^{k}=\tilde{W}$,  whereas in section \ref{martingale problem} we just used the martingale representation theorem to ensure the existence of some Wiener process which allowed to give an explicit expression of the martingale term. 
\end{proof}

\section{Appendix}
This appendix gathers several proofs or results, which although technical are not essential to the understanding of the article.   

\subsection{The damped linear wave semi-group} \label{salpha}

Let us consider the semigroup $(S_\alpha(t))$ associated with the linear wave equation with damping:
\begin{equation}
\label{ondeslin1ordre}
\left\{ \begin{array}{l} \dr_t n=\mu\\
\dr_t \mu +\alpha \mu =\dr_x^2 n,
\end{array} \right.
\end{equation}
so that the solution of the above equation with initial data $(n_0,n_1)$ can be written as $(n(t),\dr_tn(t))=S_\alpha(t)(n_0,n_1)$.  For any bounded measurable function $\varphi$ we define the Fourier multiplicator $\varphi((-\dr_x^2)^\frac{1}{2})= \mathcal{F}^{-1}\varphi(\vert \xi \vert )\mathcal{F}$ as an operator from $L^2(\ER)$ to $L^2(\ER)$.
We then note that $S_\alpha(t)$ may be defined as $S_\alpha(t)\phi = \mathcal{F}^{-1}(m_\alpha(t,\xi)\hat{\phi}(\xi))$ where: 
\begin{equation}\label{Salpha}
m_\alpha(t,\xi)=e^{-\frac{\alpha}{2}t}\begin{pmatrix}
m_\alpha^{1,1}(t,\xi) &m_\alpha^{1,2}(t,\xi) \\
m_\alpha^{2,1}(t,\xi) &m_\alpha^{2,2}(t,\xi)
\end{pmatrix}
\end{equation}
with 
\begin{equation}\label{SalphaBF}
\begin{split}
m_\alpha^{1,1}(t,\xi)=&\cosh\big(\frac{\sqrt{\alpha^2-4\xi^2}}{2}t\big)+\frac{\alpha}{\sqrt{\alpha^2-4\xi^2}}\sinh\big(\frac{\sqrt{\alpha^2-4\xi^2}}{2}t\big),\\
m_\alpha^{1,2}(t,\xi)=& \frac{2}{\sqrt{\alpha^2-4\xi^2}}\sinh\big(\frac{\sqrt{\alpha^2-4\xi^2}}{2}t\big),\\
m_\alpha^{2,1}(t,\xi)=&\big(\frac{\sqrt{\alpha^2-4\xi^2}}{2}-\frac{\alpha^2}{2\sqrt{\alpha^2-4\xi^2}}\big)\sinh\big(\frac{\sqrt{\alpha^2-4\xi^2}}{2}t\big), \\
m_\alpha^{2,2}(t,\xi)=&\cosh\big(\frac{\sqrt{\alpha^2-4\xi^2}}{2}t\big)-\frac{\alpha}{\sqrt{\alpha^2-4\xi^2}}\sinh\big(\frac{\sqrt{\alpha^2-4\xi^2}}{2}t\big),
\end{split}
\end{equation}
for $\vert \xi\vert < \frac{\alpha}{2}$, and 
\begin{equation}\label{SalphaHF}
\begin{split}
m_\alpha^{1,1}(t,\xi)=&\cos\big(\frac{\sqrt{4\xi^2-\alpha^2}}{2}t\big)+\frac{\alpha}{\sqrt{4\xi^2-\alpha^2}}\sin\big(\frac{\sqrt{4\xi^2-\alpha^2}}{2}t\big),\\
m_\alpha^{1,2}(t,\xi)=& \frac{2}{\sqrt{4\xi^2-\alpha^2}}\sin\big(\frac{\sqrt{4\xi^2-\alpha^2}}{2}t\big),\\
m_\alpha^{2,1}(t,\xi)=&\big(\frac{\sqrt{4\xi^2-\alpha^2}}{2}-\frac{\alpha^2}{2\sqrt{4\xi^2-\alpha^2}}\big)\sin\big(\frac{\sqrt{4\xi^2-\alpha^2}}{2}t\big) \\
m_\alpha^{2,2}(t,\xi)=&\cos\big(\frac{\sqrt{4\xi^2-\alpha^2}}{2}t\big)-\frac{\alpha}{\sqrt{4\xi^2-\alpha^2}}\sin\big(\frac{\sqrt{4\xi^2-\alpha^2}}{2}t\big)
\end{split}
\end{equation}
for $\vert \xi\vert > \frac{\alpha}{2}$.

Now, if $(n(t),\dr_tn(t))=S_\alpha(t)(n_0,n_1)$, with $n_0\in H^1(\ER)$ and $n_1\in L^2\cap \dot{H}^{-1}(\ER)$, then it is not difficult to see that
$$ \frac{d}{dt} (\| n (t)\|^2_{H^1} + \| \mu(t) \|^2_{L^2} + \| \dr_x^{-1} \mu(t) \|^2_{L^2}) \le 0,$$
showing that $S_\alpha$ is a contraction semi-group in $H^1(\ER) \times L^2\cap \dot{H}^{-1}(\ER)$.

Next, we state a  lemma which provides important estimates, useful in the proof of Proposition \ref{coroborneLinf}.  
This lemma requires in particular the use of homogeneous Sobolev spaces, due of the integration in time in the low frequency domain.

\begin{lemma}\label{inegalites generales}
For $k,\ell\in\EN$,  there exists a constant $C>0$ (depending on $\alpha$) such that  
\begin{equation}
\int_0^\infty \dbn{\left(S_\alpha(t)(n,m)\right)_1}_{H^k\cap\dot{H}^{-\ell}}^2dt \leqslant C\big(\dbn{n}_{H^k\cap \dot{H}^{-(\ell+1)}}^2+\dbn{m}_{H^{k-1}\cap \dot{H}^{-(\ell+1)}}^2\big),
\end{equation}
\begin{equation}
\int_0^\infty \dbn{\left(S_\alpha(t)(n,m)\right)_2}_{H^k\cap\dot{H}^{-\ell}}^2dt \leqslant C\big(\dbn{n}_{H^{k+1}\cap \dot{H}^{-(\ell+1)}}^2+\dbn{m}_{H^k\cap \dot{H}^{-(\ell+1)}}^2\big),
\end{equation} 
and for $i=1,2$, and $k\in \EN$,
\begin{equation}
\int_0^\infty\!\int_t^\infty \dbn{\left(S_\alpha(s)(0,\phi e_k)\right)_i}_{H^k\cap\dot{H}^{-\ell}}^2dsdt \leqslant C\dbn{\phi e_k}_{H^k\cap \dot{H}^{-(\ell+2)}}^2,
\end{equation}
where $\left(S_\alpha(t)(n,m)\right)_i$ denotes the i-th component of $S_\alpha(t)(n,m)$.

\end{lemma}
\begin{proof}
These  estimates are proved using similar tricks, so we choose to prove only the first one in details,  and we focus on the $H^k$ norm since 
the procedure for the $\dot{H}^{-\ell}$ norm is similar.  Thanks to the explicit form of $S_\alpha$, given in \eqref{SalphaBF} and \eqref{SalphaHF}, 
we can bound the expression $\int_0^\infty \dbn{\left(S_\alpha(t)(n,m)\right)_1}_{H^k}^2dt$ by:
\begin{align*}
C\int_0^\infty &\int_{\vert \xi\vert\leqslant \frac{\alpha}{2}}(1+\vert\xi\vert^2)^k e^{-\alpha t}\Big[ \frac{\alpha^2}{\alpha^2-4\xi^2}\sinh^2\big(\frac{\sqrt{\alpha^2-4\xi^2}}{2}t\big)\hat{n}^2(\xi) \\
&+\cosh^2\big( \frac{\sqrt{\alpha^2-4\xi^2}}{2}t\big)\hat{n}(\xi)^2  + \frac{4}{\alpha^2-4\xi^2}\sinh^2\big(\frac{\sqrt{\alpha^2-4\xi^2}}{2}t\big)\hat{m}^2(\xi)\Big]d\xi dt \\
+ \, C & \int_0^\infty \int_{\vert \xi\vert\geqslant \frac{\alpha}{2}}(1+\vert\xi\vert^2)^ke^{-\alpha t}\Big[ \frac{\alpha^2}{4\xi^2-\alpha^2}\sin^2\big(\frac{\sqrt{4\xi^2-\alpha^2}}{2}t\big)\hat{n}^2(\xi) \\
&+\cos^2\big(\frac{\sqrt{4\xi^2-\alpha^2}}{2}t\big)\hat{n}^2(\xi)   + \frac{4}{4\xi^2-\alpha^2}\sin^2\big(\frac{\sqrt{4\xi^2-\alpha^2}}{2}t\big)\hat{m}^2(\xi)\Big]d\xi dt.
\end{align*}
Now we decompose the frequency space as follows:

\vskip 0.1 in
\noindent $\bullet$ $\vert \xi\vert \leqslant\frac{\alpha}{4}$:
Here we integrate in time and get the bound: 
\begin{align*}
C\int_{\vert \xi\vert \leqslant\frac{\alpha}{4}}(1+\vert\xi\vert^2)^k\Big(\frac{1}{\alpha-\sqrt{\alpha^2-4\xi^2}}&+\frac{1}{\alpha+\sqrt{\alpha^2-4\xi^2}}\Big)
\\
&\times
\Big(\big(1+\frac{\alpha^2}{\alpha^2-4\xi^2}\big)\hat{n}^2+\frac{4}{\alpha^2-4\xi^2}\hat{m}^2\Big)d\xi  \\
= \frac12 C \alpha \int_{\vert \xi\vert \leqslant \frac{\alpha}{4}}\vert  \xi\vert^{-2} (1+\vert\xi\vert^2)^k\Big(\big( &1+\frac{\alpha^2}{\alpha^2-4\xi^2}\big)\hat{n}^2+\frac{4}{\alpha^2-4\xi^2}\hat{m}^2\Big)d\xi;
\end{align*}
since on this domain $(1+\vert \xi\vert^2)^k$ and $\frac{\alpha^2}{\alpha^2-4\xi^2}$ are bounded, we can bound this expression by $C(\dbn{n}^2_{\dot{H}^{-1}}+\dbn{m}^2_{\dot{H}^{-1}})$. Note that the $\dot{H}^{-1}$ regularity is needed due of the integration in time in this low frequency domain.  

\vskip 0.1 in
\noindent $\bullet$ $\frac{\alpha}{4}\leqslant \vert \xi\vert \leqslant\frac{\alpha}{2}$:
Here we use the inequality $\frac{1}{x}\sinh(x) \leqslant \cosh(x)$; 
since $\alpha-\sqrt{\alpha^2-4\xi^2} \geqslant (1-\frac{\sqrt{3}}{2}) \alpha$ in this domain, we deduce that $t^2 e^{-(\alpha-\sqrt{\alpha^2-4\xi^2})t}$
is uniformly integrable in time, so that we may bound here the double integral by:
$$C(\dbn{n}_{L^2}^2+\dbn{m}_{L^2}^2),$$
 due to the upper and lower bounds on $\xi$.

\vskip 0.1 in
\noindent $\bullet$
$\frac{\alpha}{2}\leqslant \vert \xi \vert \leqslant\alpha$:
We may here bound the double integral as above, using the inequality $\frac{\sin(x)}{x}\leqslant 1$. 
\vskip 0.1 in
\noindent
$\bullet$ $\vert\xi\vert\geqslant \alpha$: 
Here we integrate in time and bound the double integral by
\begin{align*}
\frac{C}{\alpha} & \int_{\vert\xi\vert\geqslant\alpha}\frac{1}{\alpha}(1+\vert\xi\vert^2)^k\Big(\big(1+\frac{\alpha^2}{4\xi^2-\alpha^2}\big)\hat{n}^2+\frac{4}{4\xi^2-\alpha^2}\hat{m}^2\Big)d\xi \\
& \leqslant C(\dbn{n}_{H^k}^2+\dbn{m}_{H^{k-1}}^2),
\end{align*}
since on this domain $\frac{1}{4\xi^2-\alpha^2}$ is bounded.

Gathering all these estimates, we obtain:
\begin{align*}
\int_0^\infty \dbn{\left(S_\alpha(t)(n,m)\right)_1}_{H^k}^2dt \leqslant  C\big(\dbn{n}_{H^k\cap\dot{H}^{-1}}^2+\dbn{m}_{H^{k-1}\cap\dot{H}^{-1}}^2\big)
\end{align*}
so that the first estimate is proved. The computations are the same for the other estimates, except that for the last one we have to integrate in time twice so that a $\dot{H}^{-2}$ regularity is needed.
\end{proof}

\subsection{Proof of Proposition \ref{borneespK}}

We start with some estimates on the martingale term $X_t+Y_t$ introduced in equations \eqref{X_t} and \eqref{Y_t}.

\begin{lemma}\label{varquadXY}
For $t\leqslant \tau^\varepsilon_\delta$, where $\tau_\delta^\varepsilon$ is defined in equation \eqref{deftaudelta}, the martingales $X_t$ and $Y_t$ defined in \eqref{X_t} and \eqref{Y_t} satisfy: 
\begin{equation}
d\langle X,X\rangle_t \leqslant C_\phi K dt,  \qquad d\langle Y,Y\rangle_t\leqslant C_\phi(1+K)dt ,   \qquad d\langle X,Y\rangle_t\leqslant C_\phi(1+K)dt,
\end{equation}
where $K$ is defined in \eqref{K}.
The constant $C_\phi$ depends on the mass $\dbn{u}_{L^2}^2$ and the Hilbert-Schmidt norm of $\phi$ in the space $H^3\cap\dot{H}^{-4}$.
\end{lemma}
\begin{proof}
From \eqref{X_t}, we get for $t\leqslant \tau_\delta^\varepsilon$:
\begin{align*}
d\langle X,X\rangle_t = & \sum_{k\in\EN}\Big( Im \int_\ER u\dr_x\bar{u}\dr_x^{-1}(\phi e_k)dx\Big)^2 dt\\
&\leqslant 4 K\dbn{u}_{L^2} \sum_{k\in\EN}\dbn{\dr_x^{-1}(\phi e_k)}_{L^\infty}^2 dt 
\leqslant C_\phi Kdt.
\end{align*}
The quadratic variation of $Y_t$ is estimated thanks to Lemma \ref{inegalites generales}:
\begin{align*}
d\langle Y,Y\rangle_t \leqslant \varepsilon C_\phi \dbn{u}_{L^4}^4 (1+\dbn{z^\varepsilon}_{H^2\cap \dot{H}^{-2}}^4+\dbn{\zeta^\varepsilon}_{H^2\cap \dot{H}^{-2}}^4)dt,
\end{align*}
so that Gagliardo-Nirenberg inequality and Proposition \ref{propprocstationnaire} yield for $\varepsilon\leqslant 1$ and $\delta$ small enough: 
$$d\langle Y,Y\rangle_t\leqslant C_\phi(1+K)dt, \; a.s.,$$
for $t\leqslant \tau^\vep_\delta$. The arguments are the same for the last quadratic variation.
\end{proof}

\begin{proof}[Proof of Proposition \ref{borneespK}]
Let us first prove that there exists $C(T)>0$ and a stopping time $\tau^\varepsilon$ such that $\mathbb P(\tau^\varepsilon \le T)$ converges to $0$ as $\varepsilon$ goes to 0 such that $$\mathbb{E}\Big[\sup_{t\leqslant\tau^\varepsilon\wedge T}(H_t^\varepsilon)^2\Big]\leqslant C(T).$$ 
We recall that, thanks to  Proposition \ref{boundpolynome},
$$\mathcal{L}^\varepsilon(H_t^\varepsilon)\leqslant\varepsilon K_t^2+BK_t+C, \; a.s.,$$ 
for all $t\leqslant \tau^\vep_\delta$, where $\delta\leqslant \frac18$ is fixed, $\tau^\vep_\delta$ is defined in \eqref{deftaudelta}, and $K$ is defined in 
\eqref{K}. Moreover, we deduce from the proof of Proposition \ref{boundpolynome} that
 $$dH^\varepsilon_t\leqslant (\varepsilon K^2+BK+C) dt +dN_t, \; a.s. \text{ for } t\leqslant \tau^\vep_\delta,$$ 
 with $N_t=X_t+Y_t$, where $X_t$ and $Y_t$ are the martingales introduced respectively in \eqref{X_t} and \eqref{Y_t}.  
In order to estimate $(H^\varepsilon)^2$, we use Proposition \ref{bornesHK} to get with the It\^o formula:
\begin{align*}
d(H_t^\varepsilon)^2=&2H_t^\varepsilon dH_t^\varepsilon +H_t^\varepsilon d\langle H^\varepsilon,H^\varepsilon\rangle_t \\
\leqslant& 2H_t^\varepsilon\left(4\varepsilon (H_t^\varepsilon)^2+BH_t^\varepsilon+C\right)dt +2H_t^\varepsilon dN_t + H_t^\varepsilon d\langle N,N\rangle_t,
\end{align*}
with possibly different constants $B$ and $C$.
Lemma \ref{varquadXY} and Young inequality then give 
$$d(H_t^\varepsilon)^2 \leqslant\left( \varepsilon(H_t^\varepsilon)^4+B(H_t^\varepsilon)^2+C\right)dt +2H_t^\varepsilon dN_t, \; a.s.$$
for $t\leqslant \tau^\vep_\delta$, with constants $B$ and $C$ still independent of $\varepsilon$. Let us introduce the stopping time 
\begin{equation}
\tau^\varepsilon=\inf\left\{t\in[0,T]; \varepsilon(H_t^\varepsilon)^2\geqslant1\right\}\wedge \tau^\varepsilon_\delta.
\end{equation}
For $t\leqslant\tau^\varepsilon$ we get 
\begin{align*}
d(H_t^\varepsilon)^2\leqslant \left((1+B)(H_t^\varepsilon)^2+C\right)dt +2H_t^\varepsilon dN_t,
\end{align*} 
from which we deduce, thanks to
Gr\"onwall Lemma, the estimate:
\begin{equation}\label{estimationH_epsilon}
{\mathbb{E}[\mathbf{1}_{[0,\tau^\vep)}(H_t^\varepsilon)^2]\leqslant \big(C(T)+\mathbb{E}[H_0^2]\big)e^{(1+B)t}}.
\end{equation}
The supremum over $[0,\tau^\varepsilon]$ is estimated thanks to a martingale inequality (see e.g.  \cite[Theorem 3.14]{DPZ14}):
\begin{align*}
\mathbb{E}\big[\sup_{t\leqslant\tau^\varepsilon}(H_t^\varepsilon)^2\big]
\leqslant  \mathbb{E}[(H_0^\vep)^2] &+C(T)+  (1+B)\mathbb{E}\Big[\int_0^{T\wedge\tau^\varepsilon}(H_s^\varepsilon)^2ds\Big]\\
 &+\mathbb{E}\Big[\big(\int_0^{T\wedge\tau^\varepsilon}(H_s^\varepsilon)^2d\langle N,N\rangle_s\big)^\frac{1}{2}\Big].
\end{align*}
Lemma \ref{varquadXY} and Proposition \ref{bornesHK} give
\begin{align*}
\mathbb{E}\Big[\big(\int_0^{T\wedge\tau^\varepsilon}(H_s^\varepsilon)^2d\langle N,N\rangle_s\big)^\frac{1}{2}\Big]
\leqslant& C_\phi \mathbb{E} \Big[ \big( \int_0^{T\wedge \tau^\vep} (H^\vep_s)^2 (1+K) ds\big)^{\frac12}\Big]\\
 \leqslant &C_\phi \mathbb{E}\Big[ \big( 1+ \sup_{t\leqslant\tau^\varepsilon}(H_t^\varepsilon)^2\big)^\frac{1}{2}
 \big(\int_0^{T\wedge\tau^\varepsilon}(H_s^\varepsilon)^2ds\big)^\frac{1}{2}\Big],
\end{align*}
so that finally 
\begin{align*}
\frac{1}{2}\mathbb{E}\Big[\sup_{t\leqslant\tau^\varepsilon}(H_t^\varepsilon)^2\Big]\leqslant&\mathbb{E}[(H_0^\vep)^2]+C(T, \phi)+ C_\phi \mathbb{E}
\Big[\int_0^{T\wedge\tau^\varepsilon}(H_s^\varepsilon)^2ds\Big],
\end{align*}
which allows to conclude, thanks to \eqref{estimationH_epsilon} and Proposition \ref{bornesHK} again, that
\begin{equation} \label{estK}
\mathbb{E}\big[\sup_{t\leqslant \tau^\varepsilon\wedge T} K^2(t)\big] \leqslant C(T).
\end{equation}

For the second estimate, we note that for  $t\leqslant\tau^\varepsilon$,
\begin{align*}
H^\varepsilon_t+\alpha\int_0^t\dbn{\sqrt{\varepsilon}V}_{L^2}^2ds = \int_0^t (\mathcal{L}^\varepsilon(H^\varepsilon_s)+\dbn{\sqrt{\varepsilon}V}_{L^2}^2) \, ds + (X_t+Y_t) + H^\varepsilon_0,
\end{align*}
where $X_t$ and $Y_t$ are the martingales defined in equation \eqref{X_t} and \eqref{Y_t}.
Applying Proposition \ref{boundpolynome}, then Proposition \ref{bornesHK}, we deduce:
\begin{align*}
\frac{1}{2}K(t)+\alpha\int_0^t\dbn{\sqrt{\varepsilon}V}_{L^2}^2ds \leqslant \int_0^t (\varepsilon K^2+BK+C)ds +(X_t+Y_t) +C,
\end{align*} 
so that the second estimate of Proposition \ref{borneespK} simply follows from \eqref{estK} and the martingale inequality of
\cite[Theorem 3.14]{DPZ14}.

It remains to prove the convergence of $\tau^\varepsilon$ as $\varepsilon$ 
goes to $0$, but this is a simple consequence of Markov inequality, since
\begin{align*}
\mathbb{P}(\tau^\varepsilon<T) \leqslant \mathbb{P}\big(\sup_{t\leqslant\tau^\varepsilon}(H_t^\varepsilon)^2\geqslant\frac{1}{\varepsilon}\big) 
\leqslant \varepsilon C(T).
\end{align*}
This concludes the proof of Proposition \ref{borneespK}.
\end{proof}

\subsection{Proof of Lemma \ref{k(x,y)}}

We recall that the aim here is to compute 
$$k(x,y)=\mathbb{E}_\nu\left[z(x)\MM^{-1}z(y)+z(y)\MM^{-1}z(x)\right],$$
with, thanks to Lemma \ref{inverse1}, $\MM^{-1}z=(\dr_x^2)^{-1}\zeta+\alpha(\dr_x^2)^{-1}z$.  We start by computing the term $\mathbb{E}_\nu\left[z(x)(\dr_x^2)^{-1}z(y)\right]$.  According to  equation \eqref{esp_nu},  this term is equal to $$\sum_{k\in\EN}\int_0^\infty (S_\alpha(t)(0,\phi e_k))_1(x)(\dr_x^2)^{-1} (S_\alpha(t)(0,\phi e_k))_1(y)dt$$ 
with, according to \eqref{Salpha}, \eqref{SalphaBF} and \eqref{SalphaHF},
\begin{align*}
 (S_\alpha(t)(0,\phi e_k))_1(x) =& 2e^{-\frac{\alpha}{2}t}\int_{\vert \xi\vert\leqslant\frac{\alpha}{2}}\sinh\Big(\frac{\sqrt{\alpha^2-4\xi^2}}{2}t\Big)\frac{\widehat{\phi e_k}(\xi)}{\sqrt{\alpha^2-4\xi^2}}\,e^{ix\xi}d\xi \\
 &+ 2e^{-\frac{\alpha}{2}t}\int_{\vert \xi\vert\geqslant\frac{\alpha}{2}}\sin\Big(\frac{\sqrt{4\xi^2-\alpha^2}}{2}t\Big)\frac{\widehat{\phi e_k}(\xi)}{\sqrt{4\xi^2-\alpha^2}}\,e^{ix\xi}d\xi.
\end{align*}
Thus, we deduce:
\begin{align*}
\mathbb{E}_\nu\left[z(x)(\dr_x^2)^{-1}z(y)\right]=4\sum_{k\in\EN}&\int_0^\infty e^{-\alpha t} \Big[\int_{\vert \xi\vert\leqslant\frac{\alpha}{2}}\sinh\Big(\frac{\sqrt{\alpha^2-4\xi^2}}{2}t\Big)\frac{\widehat{\phi e_k}(\xi)}{\sqrt{\alpha^2-4\xi^2}}e^{ix\xi}d\xi \\
 &\qquad\qquad + \int_{\vert \xi\vert\geqslant\frac{\alpha}{2}}\sin\Big(\frac{\sqrt{4\xi^2-\alpha^2}}{2}t\Big)\frac{\widehat{\phi e_k}(\xi)}{\sqrt{4\xi^2-\alpha^2}}e^{ix\xi}d\xi\Big] \\
 &\qquad\times \Big[\int_{\vert \eta\vert\leqslant\frac{\alpha}{2}}\sinh\Big(\frac{\sqrt{\alpha^2-4\eta^2}}{2}t\Big)\frac{\widehat{(\dr_x^2)^{-1}\phi e_k}(\eta)}{\sqrt{\alpha^2-4\eta^2}}e^{iy\eta}d\eta \\
 &\qquad\qquad+ \int_{\vert \eta\vert\geqslant\frac{\alpha}{2}}\sin\Big(\frac{\sqrt{4\eta^2-\alpha^2}}{2}t\Big)\frac{\widehat{(\dr_x^2)^{-1}\phi e_k}(\eta)}{\sqrt{4\eta^2-\alpha^2}}e^{iy\eta}d\eta\Big]dt  \\
 = I+ II &+III+IV.
\end{align*}
We focus on the computation of the low frequency product: 
\begin{align*}
I=2\sum_{k\in\EN}&\int_{\vert\xi\vert\leqslant\frac{\alpha}{2}}\int_{\vert\eta\vert\leqslant\frac{\alpha}{2}} \frac{\widehat{\phi e_k}(\xi)}{\sqrt{\alpha^2-4\xi^2}} \, \frac{\widehat{(\dr_x^2)^{-1}\phi e_k}(\eta)}{\sqrt{\alpha^2-4\eta^2}}e^{i(x\xi+y\eta)}\\
& \times \int_0^\infty e^{-\alpha t}\Big[\cosh\Big(\frac{\sqrt{\alpha^2-4\xi^2}}{2}t+\frac{\sqrt{\alpha^2-4\eta^2}}{2}t\Big)\\
&\qquad \qquad
-\cosh \Big(\frac{\sqrt{\alpha^2-4\xi^2}}{2}t-\frac{\sqrt{\alpha^2-4\eta^2}}{2}t\Big)\Big] \, dt\, d\eta \, d\xi.
\end{align*}
For fixed $k, \xi, \eta$,  the time integral is equal to 
\begin{align*}
\frac{1}{2}&\Big[\frac{1}{\alpha-\frac{1}{2}\sqrt{\alpha^2-4\xi^2}-\frac{1}{2}\sqrt{\alpha^2-4\eta^2}}+\frac{1}{\alpha+\frac{1}{2}\sqrt{\alpha^2-4\xi^2}+\frac{1}{2}\sqrt{\alpha^2-4\eta^2}} \\
&-\frac{1}{\alpha-\frac{1}{2}\sqrt{\alpha^2-4\xi^2}+\frac{1}{2}\sqrt{\alpha^2-4\eta^2}}-\frac{1}{\alpha+\frac{1}{2}\sqrt{\alpha^2-4\xi^2}-\frac{1}{2}\sqrt{\alpha^2-4\eta^2}}\Big] \\[0.25cm]
&=\frac{\alpha\sqrt{\alpha^2-4\xi^2}\sqrt{\alpha^2-4\eta^2}}{(\xi^2-\eta^2)^2+2\alpha^2(\xi^2+\eta^2)},
\end{align*} 
so that 
\begin{align*}
I= 2\alpha\sum_{k\in\EN}\int_{\vert\xi\vert\leqslant\frac{\alpha}{2}}\int_{\vert\eta\vert\leqslant\frac{\alpha}{2}}\frac{\widehat{\phi e_k}(\xi)\widehat{(\dr_x^2)^{-1}\phi e_k}(\eta)}{(\xi^2-\eta^2)^2+2\alpha^2(\xi^2+\eta^2)}e^{i(x\xi+y\eta)}d\eta \,d\xi.
\end{align*}
Siimilar computations for the other terms lead to:
\begin{equation}
\mathbb{E}_\nu\left[z(x)(\dr_x^2)^{-1}z(y)\right]=\sum_{k\in\EN}\int_\ER\int_\ER K_1(\xi,\eta)\widehat{\phi e_k}(\xi)\widehat{\phi e_k}(\eta)e^{i(x\xi+y\eta)}d\eta \,d\xi, 
\end{equation}
with 
\begin{equation}\label{defK_1}
K_1(\xi,\eta)=-\frac{1}{\eta^2}\times\frac{2\alpha}{(\xi^2-\eta^2)^2+2\alpha^2(\xi^2+\eta^2)}.
\end{equation}
The computation of  $\mathbb{E}_\nu\left[z(x)(\dr_x^2)^{-1}\zeta(y)\right]$ is very similar so that we do not write it in details. We obtain 
\begin{equation}
\begin{split}
\mathbb{E}_\nu\left[z(x)(\dr_x^2)^{-1}\zeta(y)\right]=&-\frac{\alpha}{2}\mathbb{E}_\nu\left[z(x)(\dr_x^2)^{-1}z(y)\right] \\&+ \sum_{k\in\EN}\int_\ER\int_\ER K_2(\xi,\eta)\widehat{\phi e_k}(\xi)\widehat{\phi e_k}(\eta)e^{i(x\xi+y\eta)}d\eta \, d\xi, 
\end{split}
\end{equation}
with 
\begin{equation}\label{defK_2}
K_2(\xi,\eta)=-\frac{1}{\eta^2}\times\frac{\alpha^2+\xi^2-\eta^2}{(\xi^2-\eta^2)^2+2\alpha^2(\xi^2+\eta^2)}.
\end{equation}
Finally, defining $K(\xi,\eta)=\frac{\alpha}{2}K_1(\xi,\eta)+K_2(\xi,\eta)$    we obtain 
\begin{equation}
\begin{split}
\mathbb{E}_\nu [z(x)&\MM^{-1}z(y)+z(y)\MM^{-1}z(x)]\\
& =\sum_{k\in\EN}\int_\ER\int_\ER (K(\xi,\eta)+K(\eta,\xi))\widehat{\phi e_k}(\xi)\widehat{\phi e_k}(\eta)e^{i(x\xi+y\eta)}d\eta \, d\xi \\
&=-\sum_{k\in\EN}\int_\ER\int_\ER\frac{1}{\xi^2\eta^2}\widehat{\phi e_k}(\xi)\widehat{\phi e_k}(\eta)e^{i(x\xi+y\eta)}d\eta \, d\xi
\end{split}
\end{equation}
which gives the result of Lemma \ref{k(x,y)}.

\bibliographystyle{plain}
\bibliography{biblio1}

\end{document}